\documentclass[a4paper,11pt,reqno]{amsart}

\usepackage{amsmath}
\usepackage{amsthm}
\usepackage{amssymb}
\usepackage{graphicx}
\usepackage{epsfig}
\usepackage[pdftex]{color}
 \definecolor{MyRed}{rgb}{0.9,0,0}
 \definecolor{MyGreen}{rgb}{0,0.9,0}
 \definecolor{MyBlue}{rgb}{0,0,0.9}
\usepackage[pdftex]{hyperref}
\hypersetup{breaklinks=true,colorlinks=true,citecolor=MyGreen,urlcolor=MyRed,linkcolor=MyBlue,pdfpagemode=UseNone}
\theoremstyle{plain}

\newtheorem{theorem}{Theorem}

\newtheorem*{thm*}{Theorem}
\newtheorem*{theorem*}{Theorem}

\newtheorem*{claim*}{Claim}
\newtheorem{lemma}[theorem]{Lemma}

\newtheorem{conjecture}[theorem]{Conjecture}

\theoremstyle{definition}

\newcommand{\eps}{\varepsilon}
\renewcommand{\leq}{\leqslant}
\renewcommand{\geq}{\geqslant}

\def\Kuhn{K{\"u}hn}
\def\Szemeredi{Szemer\'edi}

\def\COMMENT#1{}
\def\E{\mathbb E}
\def\P{\mathbb P}

\begin{document}

\title{Edge-disjoint Hamilton cycles in graphs}
\author{Demetres Christofides, Daniela \Kuhn\ and Deryk Osthus}
\thanks {D.~Christofides and D.~Osthus were supported by the EPSRC, grant no.~EP/E02162X/1.
D.~K\"uhn was supported by the ERC, grant no.~258345.}
\date{\today}
\subjclass[2000]{05C35,05C45,05C70,05D40.}
\keywords{Hamilton cycles; graph decompositions; regularity lemma; probabilistic methods}
\begin{abstract}
In this paper we give an approximate answer to a question of
Nash-Williams from 1970: we show that for every $\alpha > 0$, every sufficiently large graph
on $n$ vertices with minimum degree at least $(1/2 + \alpha)n$
contains at least $n/8$ edge-disjoint Hamilton cycles.
More generally, we
give an asymptotically best possible answer for the number
of edge-disjoint Hamilton cycles that a graph $G$ with minimum
degree $\delta$ must have. We also prove
an approximate version of another long-standing conjecture of
Nash-Williams: we show that for every $\alpha
> 0$, every (almost) regular and sufficiently large graph on $n$
vertices with minimum degree at least $(1/2 + \alpha)n$ can be
almost decomposed into edge-disjoint Hamilton cycles.
\end{abstract}
\maketitle

\section{Introduction}

Dirac's theorem~\cite{Dirac52} states that every graph on $n \geq 3$
vertices of minimum degree at least $n/2$ contains a Hamilton cycle.
The theorem is best possible since there are graphs of minimum
degree at least $\lfloor (n-1)/2 \rfloor$ which do not contain any
Hamilton cycle.

Nash-Williams~\cite{Nash-Williams71a} proved the surprising result that the conditions of
Dirac's theorem, despite being best possible,
even guarantee the existence of many edge-disjoint Hamilton cycles.

\begin{theorem}[\cite{Nash-Williams71a}] \label{Nash-Williams - 5n/224}
Every graph on $n$ vertices of minimum degree at least $n/2$
contains at least $\lfloor 5n/224 \rfloor$ edge-disjoint Hamilton
cycles.
\end{theorem}

Nash-Williams~\cite{Nash-Williams70,Nash-Williams71a,Nash-Williams71b}
asked whether the above bound on the number of Hamilton cycles can be
improved. Clearly we cannot expect more than $\lfloor (n+1)/4 \rfloor$ edge-disjoint Hamilton
cycles and
Nash-Williams~\cite{Nash-Williams70} initially conjectured that one
might be able to achieve this. However, soon afterwards, it was
pointed out by Babai (see~\cite{Nash-Williams70}) that this
conjecture is false.  Babai's idea was carried further by
Nash-Williams~\cite{Nash-Williams70} who gave an example of a graph
on $n=4m$ vertices with minimum degree $2m$ having at most $\lfloor
(n+4)/8 \rfloor$ edge-disjoint Hamilton cycles. Here is a
similar example having at most $\lfloor (n+2)/8 \rfloor$
edge-disjoint Hamilton cycles: Let $A$ be an empty graph on $2m$
vertices, $B$ a graph consisting of $m+1$ disjoint edges and let $G$ be the
graph obtained from the disjoint union of $A$ and $B$ by adding all
possible edges between $A$ and $B$. So $G$ is a graph on $4m+2$
vertices with minimum degree $2m+1$.
Observe that any Hamilton cycle of $G$ must use at least 2 edges
from $B$ and thus $G$ has at most $\lfloor (m+1)/2 \rfloor$
edge-disjoint Hamilton cycles.
We will prove that this example is asymptotically best possible.

\begin{theorem}\label{Minimum Degree n/2}
For every $\alpha > 0$ there is an integer $n_0$ so that every graph
on $n \geq n_0$ vertices of minimum degree at least $(1/2 +
\alpha)n$ contains at least $n/8$ edge-disjoint Hamilton cycles.
\end{theorem}

Nash-Williams~\cite{Nash-Williams70,Nash-Williams71b} pointed out that the
construction described above depends heavily on the graph being
non-regular. He thus conjectured~\cite{Nash-Williams71b} the
following, which if true is clearly best possible.

\begin{conjecture}[\cite{Nash-Williams71b}]\label{Nash-Williams - Decomposition}
Let $G$ be a $d$-regular graph on at most $2d$ vertices. Then $G$
contains $\lfloor d/2 \rfloor$ edge-disjoint Hamilton cycles.
\end{conjecture}

The conjecture was also raised independently by
Jackson~\cite{Jackson79}. For complete graphs, its truth follows
from a construction of Walecki (see e.g.~\cite{abs,lucas}). The best result towards this conjecture is the
following result of Jackson~\cite{Jackson79}.

\begin{theorem}[\cite{Jackson79}]
Let $G$ be a $d$-regular graph on $14 \leq n \leq 2d+1$ vertices.
Then $G$ contains $\lfloor (3d-n+1)/6 \rfloor$ edge-disjoint
Hamilton cycles.
\end{theorem}

In this paper we prove an approximate version of Conjecture~\ref{Nash-Williams -
Decomposition}.

\begin{theorem}\label{Regular}
For every $\alpha > 0$ there is an integer $n_0$ so that every
$d$-regular graph on $n \geq n_0$ vertices with $d \geq (1/2 +
\alpha)n$ contains at least $(d - \alpha n)/2$ edge-disjoint
Hamilton cycles.
\end{theorem}

In fact, we will prove the following more general result which states
that Theorem~\ref{Regular} is true for almost regular graphs as
well. Note that the construction showing that
one cannot achieve more than $\lfloor (n+2)/8\rfloor$ edge-disjoint Hamilton cycles under
the conditions of Dirac's theorem is almost regular.
However in the following result we also demand that the minimum
degree is a little larger than $n/2$.

\begin{theorem}\label{Almost regular}
There exists $\alpha_0>0$ so that for every $0<\alpha \le \alpha_0$
there is an integer $n_0$ for which every graph
on $n \geq n_0$ vertices with minimum degree $\delta \geq (1/2 +
\alpha)n$ and maximum degree $\Delta \leq \delta + \alpha^2 n/5$
contains at least $(\delta - \alpha n)/2$ edge-disjoint Hamilton
cycles.
\end{theorem}
Frieze and Krivelevich~\cite{Frieze&Krivelevich05} proved that the above results hold
if one also knows that the graph is quasi-random (in which case one can drop the
condition on the minimum degree).
So in particular, it follows that a binomial random graph $G_{n,p}$ with constant edge probability $p$
can `almost' be decomposed into Hamilton cycles with high probability.
For such $p$, it is still an open question whether one can improve this to show
that with high probability the number of edge-disjoint Hamilton cycles is exactly half the minimum degree
-- see e.g.~\cite{Frieze&Krivelevich05} for a further discussion.
Our proof makes use of the ideas in~\cite{Frieze&Krivelevich05}.

Finally, we answer the question of what happens if we have a better bound on the minimum
degree than in Theorem~\ref{Minimum Degree n/2}.
The following result approximately describes how the number of edge-disjoint Hamilton cycles
guaranteed in $G$ gradually approaches $\delta(G)/2$ as $\delta(G)$ approaches $n-1$.
\begin{theorem}\label{Minimum Degree}
$ $
\begin{itemize}
\item[(i)] For all positive integers $\delta,n$ with $n/2 < \delta <
n$, there is a graph $G$ on $n$ vertices with minimum degree
$\delta$ such that $G$ contains at most
\[
\frac{\delta + 2 + \sqrt{n(2 \delta - n)}}{4}
\]
edge-disjoint Hamilton cycles.
\item[(ii)] For every $\alpha > 0$, there is a positive integer
$n_0$ so that every graph on $n \geq n_0$ vertices of minimum degree
$\delta \geq (1/2 + \alpha)n$ contains at least
\begin{equation}\label{eqbound}
\frac{\delta - \alpha n +  \sqrt{n(2 \delta - n)}}{4}
\end{equation}
edge-disjoint Hamilton cycles.
\end{itemize}
\end{theorem}

Observe that Theorem~\ref{Minimum Degree n/2} is an immediate
consequence of Theorem~\ref{Minimum Degree}(ii). In
Section~\ref{sec:bestposs} we will give a simple construction which
proves Theorem~\ref{Minimum Degree}(i). This construction also
yields an analogue of Theorem~\ref{Minimum Degree} for $r$-factors,
where $r$ is even: Clearly, Theorem~\ref{Minimum Degree}(ii) implies
the existence of an $r$-factor for any even $r$ which is at most
twice the bound in~\eqref{eqbound}. The construction in
Section~\ref{sec:bestposs} shows that this is essentially best
possible. The question of which conditions on a graph guarantee an
$r$-factor has a huge literature, see the survey by Plummer for a
recent overview~\cite{plummer}.

It turns out that the proofs of Theorems~\ref{Almost
regular} and~\ref{Minimum Degree}(ii) are very similar and
we will thus prove these results simultaneously. In Section~3 we give an overview
of the proof. In Section~4 we introduce some notation and also some
tools that we will need in the proofs of Theorems~\ref{Almost
regular} and~\ref{Minimum Degree}(ii). We prove these
theorems in Section~5.

Another long-standing conjecture in the area is due to Kelly (see e.g.~\cite{moon}).
It states that any regular tournament can be decomposed into edge-disjoint Hamilton cycles.
Very recently, an approximate version of this conjecture was proved in~\cite{Kuhn&Osthus&Treglown}.
The basic proof strategy is common to both papers. So we hope that the
proof techniques will also be useful for further decomposition problems.

%%%%%%%%%%%%%%%%%%%%%%%%%%%%%%%%%%%%%%%%%%%%%%%%%%%%%%%%%%%%%%%%%%%%%%%%%%%%%%%%%%%%%%%%%%%%%%%%%%%%%%%%%%%%%%%%%%%%%%%%%%%%%%%%%%%%%%%%%%%%%%%%%%%%%%%%%
\section{Proof of Theorem~\ref{Minimum Degree}(i)}\label{sec:bestposs}

If $\delta = n-1$, then $K_n$ contains at most
\[
\frac{n-1}{2} = \frac{n + (n-2)}{4} < \frac{n +1+ \sqrt{n(n-2)}}{4} =
\frac{\delta + 2 + \sqrt{n(2\delta - n)}}{4}
\]
edge-disjoint Hamilton cycles. So from now on we will assume that
$\delta \leq n-2$.

The construction of the graph $G$ is very similar to the
construction in the introduction showing that we might not have more
than $\lfloor (n+2)/8\rfloor$ edge-disjoint Hamilton cycles. Here, $G$ will be the
disjoint union of an empty graph $A$ of size $n - \Delta$, and a
$(\delta + \Delta - n)$-regular graph $B$ on $\Delta$ vertices,
together with all edges between $A$ and $B$ (see
Figure~\ref{Example2}). Such a graph $B$ exists if for example
$\Delta$ is even (see e.g.~\cite[Problem 5.2 ]{Lovasz07}).

\begin{figure}[h]
\includegraphics[scale=1]{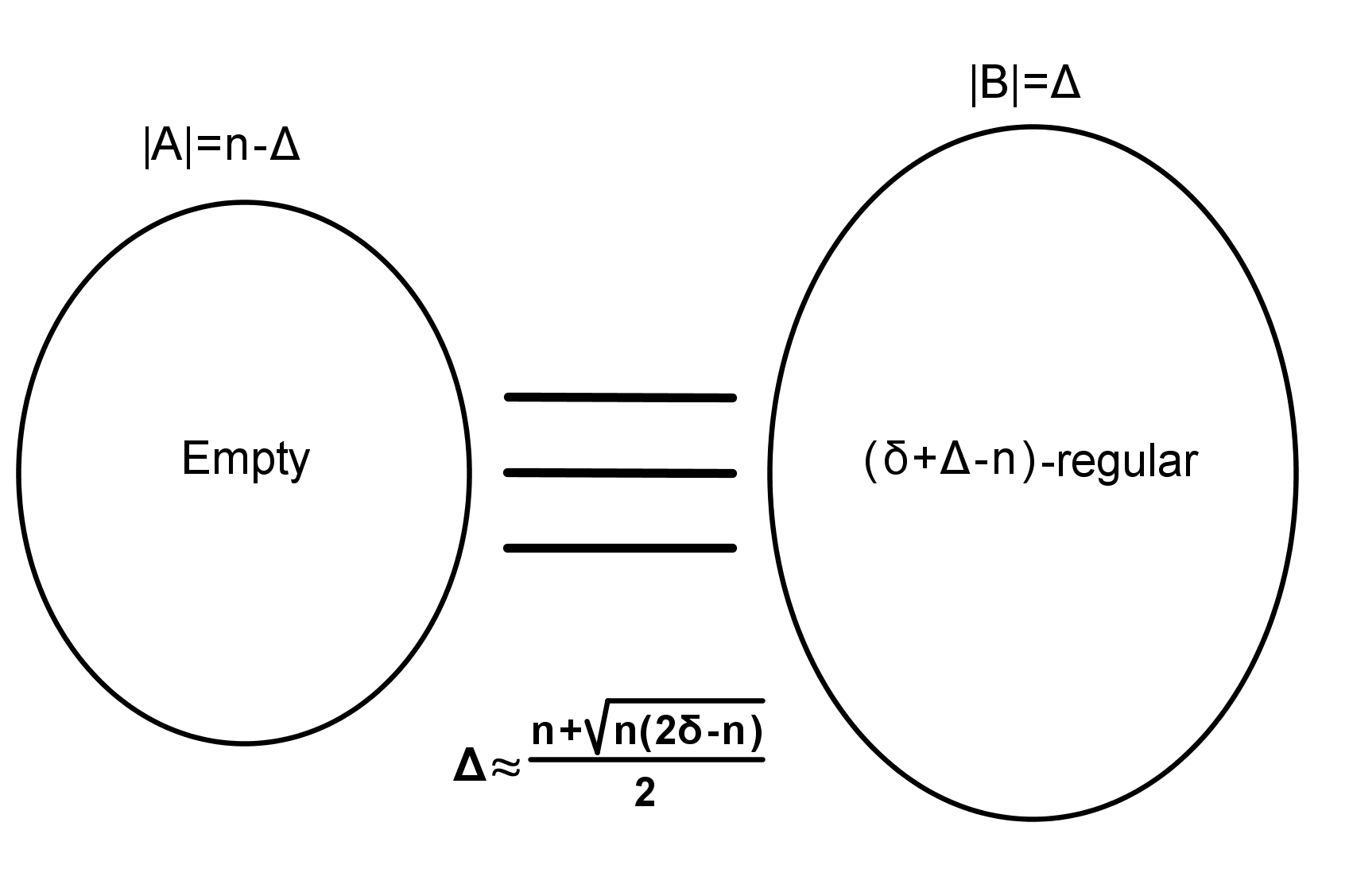}
\caption{A graph $G$ on $n$ vertices with minimum degree at least
$\delta > n/2$ having at most $\frac{\delta + 2 +
\sqrt{n(2\delta-n)}}{2}$ edge-disjoint Hamilton cycles.}
\label{Example2}
\end{figure}

The value of $\Delta$ will be chosen later. At the moment we will
only demand that $\Delta$ is an even integer satisfying $\delta \leq
\Delta \leq n-1$. Observe that $G$ is a graph on $n$ vertices with
minimum degree $\delta$ and maximum degree $\Delta$. We claim that
$G$ cannot contain more than $\frac{\Delta (\delta + \Delta -
n)}{2(2 \Delta - n)}$ edge-disjoint Hamilton cycles. In fact, we
claim that it can only contain an $r$-factor if $r \le \frac{\Delta
(\delta + \Delta - n)}{2 \Delta - n}$. Indeed, given any
$r$-factor $H$ of $G$, since $e_H(A,B) = \sum_{v \in A}d_H(v) = r (n
- \Delta)$, we deduce that
\[
r \Delta = \sum_{v \in B}d_H(v) \leq \Delta(\delta + \Delta - n) +
r(n - \Delta)
\]
from which our claim follows. It remains to make a judicious choice
for $\Delta$ and to show that it implies the result. One can check that
$\frac{n + \sqrt{n(2\delta - n)}}{2}$ minimizes $f(x) = x(\delta + x
- n)/(2x-n)$ in $[\delta,n]$. (This is only used as a heuristic and
it is not needed in our argument.) It can be also checked that since
$\delta \leq n-2$ we have $\delta \leq \frac{n + \sqrt{n(2\delta -
n)}}{2} < n-1$. Indeed, the first inequality holds if and only if
$(2\delta - n)^2 \leq n(2\delta - n)$ which is true as $n/2\leq \delta \leq
n$ and the second inequality holds since $$\frac{n + \sqrt{n(2\delta
- n)}}{2} \leq \frac{n + \sqrt{n^2 - 4n}}{2} < \frac{n + (n-2)}{2} =
n-1.$$
We define $\Delta = \frac{n + \sqrt{n(2\delta - n)}}{2} +
\eps$, where $\eps$ is chosen so that $|\eps| \leq 1$ and $\Delta$
is an even integer satisfying $\delta \leq \Delta \leq n-1$. We
claim that this value of $\Delta$ gives the desired bound. To see
this, recall that if $G$ contains an $r$-factor, then we must have
\[
r \le \frac{\Delta (\delta + \Delta - n)}{2 \Delta - n} =
\frac{\delta}{2} + \frac{n \delta/2}{2\Delta - n} -
\frac{\Delta(n-\Delta)}{2\Delta - n}
\]
and that
\begin{align*}
\Delta(n - \Delta) &= \left(\frac{n}{2} +
\left(\frac{\sqrt{n(2\delta - n)}}{2} + \eps\right) \right)
\left(\frac{n}{2} -
\left(\frac{\sqrt{n(2\delta - n)}}{2} + \eps\right) \right)\\
& = \frac{n^2}{4} - \frac{n(2\delta - n)}{4} -
\eps\sqrt{n(2\delta - n)} - \eps^2
= \frac{n^2 - n\delta}{2} - \eps\sqrt{n(2\delta - n)} - \eps^2.
\end{align*}
Thus
\[
r \leq \frac{\delta}{2} + \frac{n(2\delta - n) + 2\eps
\sqrt{n(2\delta - n)}}{2(2\Delta - n)} + \frac{\eps^2}{2\Delta - n}.
\]
Since also
$
(2\Delta - n)\sqrt{n(2\delta-n)} = n(2\delta-n) +
2\eps\sqrt{n(2\delta - n)},
$
we deduce that
\[
r \leq \frac{\delta + \sqrt{n(2\delta-n)}}{2} +
\frac{\eps^2}{2\Delta - n} \leq \frac{\delta + 2 +
\sqrt{n(2\delta-n)}}{2},
\]
as required.

%%%%%%%%%%%%%%%%%%%%%%%%%%%%%%%%%%%%%%%%%%%%%%%%%%%%%%%%%%%%%%%%%%%%%%%%%%%%%%%%%%%%%%%%%%%%%%%%%%%%%%%%%%%%%%%%%%%%%%%%%%%%%%%%%%%%%%%%%%%%%%%%%%%%%%%%%

\section{Proof overview of the main theorems}

In the overview we will only discuss the case in which $G$ is
regular, say of degree $\lambda n$ with $\lambda >1/2$.\COMMENT{$d$ is used for the
regularity density.} The other cases are similar and in fact will be
treated simultaneously in the proof itself.
We begin by defining additional constants such that

\[0<  \eps \ll \beta  \ll \gamma \ll 1.  \]

By applying the Regularity Lemma to $G$, we obtain a partition of $G$ into clusters
$V_1,\ldots,V_k$ and an exceptional set~$V_0$. Moreover, most pairs of clusters span an $\eps$-regular
(i.e.~quasi-random) bipartite graph.
It turns out that for our purposes the `standard' reduced graph defined on the clusters does not capture enough information
about the original graph $G$.
So we will instead work with the multigraph $R$ on vertex set
$\{V_1,\ldots,V_k\}$ in which there are exactly $\ell_{ij} :=
\lfloor d(V_i,V_j)/\beta \rfloor$ multiple edges between the
vertices $V_i$ and $V_j$ of $R$ (provided that the pair $(V_i,V_j)$ is $\eps$-regular).
Here $d(V_i,V_j)$ denotes the density of
the bipartite subgraph induced by $V_i$ and $V_j$. Then $R$ is almost regular, with
all degrees close to $\lambda k/\beta$. In particular, we can use
Tutte's $f$-factor theorem (see Theorem~\ref{Factor Theorem}(ii)) to
deduce that $R$ contains an $r$-regular submultigraph $R'$ where $r$
is still close to $\lambda k/\beta$. By Petersen's theorem, $R'$ can be
decomposed into 2-factors and by splitting the clusters if necessary
we may assume that $R'$ can be decomposed into 2-factors such that
every cycle has even length. In particular, $R'$ can be decomposed
into $r$ perfect matchings, say $M_1, \ldots , M_r$.

We now partition (most of) the edges of $G$ in such a way that each matching edge
is assigned roughly the same number of edges of $G$. More precisely,
given two
adjacent clusters $U,V$ of $R$, the edge set $E_{G}(U,V)$ can
be decomposed into $\ell_{ij}$ bipartite graphs so that each is $\eps$-regular with density close to
$\beta$. These $\ell_{ij}$ regular pairs correspond to
the $\ell_{ij}$ edges in $R$ between $U$ and $V$. Thus, for each matching
$M_i$, we can define a subgraph $G_i$ of $G$ such that all $G_i$'s
are edge-disjoint and they consist of a union of $k':=k/2$  pairs of clusters which are $\eps$-regular
of density about $\beta$, together with the exceptional set $V_0$.
Let $m$ denote the size of a cluster.
By moving some additional vertices to the exceptional set, we may assume that
for every such pair of clusters of $G_i$, all vertices have degree close to $\beta m$.
So for each
$i$, we now have a set $V_{0i}$ consisting of the exceptional set
$V_0$ together with the vertices moved in the previous step.
For each $G_i$ we will aim to find close to
$\beta m/2$ edge-disjoint Hamilton cycles consisting mostly of edges
of $G_i$ and a few further edges which do not belong to any of the $G_i$.

Because $G$ may not have many edges which do not belong to any of
the $G_i$, (in fact it may have none) before proceeding we extract
random subsets of edges from each $G_i$ to get disjoint subgraphs
$H_1,H_2$ and $H_3$ of $G$ each of density about $\gamma$ which
satisfy several other useful properties as well. Moreover, each pair
of clusters of $G_i$ corresponding to an edge of $M_i$ will still be
super-regular of density almost $\beta$. Each of the subgraphs
$H_1,H_2$ and $H_3$ will be used for a different purpose in the
proof.

$H_1$ will be used to connect the vertices of each $V_{0i}$ to
$G_i \setminus V_{0i}$ so that the vertices of $V_{0i}$ have
almost $\beta m$ neighbours in $V(G_i) \setminus V_{0i}$.\COMMENT{In fact we are
not just using $H_1$ here. We will also use edges which do not
belong to any $G_i$ or $H_j$.} Moreover the edges added to $G_i$
will be well spread-out in the sense that no vertex of $G_i
\setminus V_{0i}$ will have large degree in $V_{0i}$.
So every vertex of $G_i$ now has degree close to $\beta m$.

Next, our aim is to find an $s$-regular spanning subgraph $S_i$ of $G_i$
with $s$ close to $\beta m$. In order to achieve this, it turns out that
we will first need to add some edges to $G_i$ between pairs of clusters which do not
correspond to edges of $M_i$. We will take these from $H_2$.

We may assume that the degree of $S_i$ is even and thus by
Petersen's theorem it can be decomposed into 2-factors. It will
remain to use the edges of $H_3$ to transform each of these 2-factors into a
Hamilton cycle. Several problems may arise here. Most notably, the
number of edges of $H_3$ we will need in order to transform a given
2-factor $F$ into a Hamilton cycle will be proportional to the
number of cycles of $F$. So if we have a linear number of
2-factors $F$ which have a linear number of cycles, then we will
need to use a quadratic number of edges from $H_3$ which would destroy most of its useful properties.
However, a result from~\cite{Frieze&Krivelevich05} based on estimating
the permanent of a matrix implies that the average number of cycles in a $2$-factor
of $S_i$ is $o(n)$. We will apply a variant of this result proved
in~\cite{randmatch,Kuhn&Osthus&Treglown}. So we can assume that our $2$-factors
have $o(n)$ cycles.

To complete the proof we will consider a random  partition of the graph $H_3$ into
subgraphs $H_{3,1}, \ldots , H_{3,r}$, one for each graph $G_i$.
We will
use the edges of $H_{3,i}$ to transform all 2-factors of $S_i$ into
Hamilton cycles. We will achieve this by considering each 2-factor $F$
successively. For each $F$, we will use the rotation-extension technique to
successively merge its cycles. Roughly speaking, this means that we obtain a path $P$
with endpoints $x$ and $y$ (say) by removing a suitable edge of a cycle of $F$.
If $F$ is not a Hamilton cycle and $H_{3,i}$ has an edge from $x$ or $y$ to another
cycle $C$ of $F$, and we can extend
$P$ to a path containing all vertices of $C$ as well. We continue in this
way until in $H_{3,i}$ both endpoints of $P$ have all their neighbours on $P$.
We can then use this to find a cycle $C'$ containing precisely all vertices of $P$.
In the final step, we make use (amongst others) of the quasi-randomness of the
bipartite graphs which form $H_{3,i}$.

%%%%%%%%%%%%%%%%%%%%%%%%%%%%%%%%%%%%%%%%%%%%%%%%%%%%%%%%%%%%%%%%%%%%%%%%%%%%%%%%%%%%%%%%%%%%%%%%%%%%%%%%%%%%%%%%%%%%%%%%%%%%%%%%%%%%%%%%%%%%%%%%%%%%%%%%%

\section{Notation and Tools}

\subsection{Notation}

Given vertex sets $A$ and $B$ in a graph $G$, we write $E_G(A,B)$
for the set of all edges $ab$ with $a \in A$ and $b \in B$ and put
$e_G(A,B) = |E_G(A,B)|$. We write $(A,B)_G$ for the bipartite subgraph of $G$
whose vertex classes are $A$ and $B$ and whose set of edges is $E_G(A,B)$.
We drop the subscripts if this is unambiguous. Given a set $E'\subseteq E_G(A,B)$,
we also write $(A,B)_{E'}$ for the bipartite subgraph of~$G$ whose vertex classes
are $A$ and $B$ and whose set of edges is~$E'$.
Given a vertex $x$ of $G$ and a set $A\subseteq V(G)$, we write $d_A(x)$
for the number of neighbours of $x$ in~$A$.

To prove Theorems~\ref{Almost regular} and~\ref{Minimum Degree}(ii)
it will be convenient to work with multigraphs instead of just
(simple) graphs. All multigraphs considered in this paper will be
without loops.

We write $a = b \pm c$ to mean that the real numbers $a,b,c$ satisfy
$|a-b| \leq c$.
To avoid unnecessarily complicated calculations we will sometimes
omit floor and ceiling signs and treat large numbers as if they were
integers. We will also sometimes treat large numbers as if they were even
integers.

\subsection{Chernoff Bounds}

Recall that a Bernoulli random variable with parameter $p$ takes the
value 1 with probability $p$ and the value 0 with probability $1-p$.
We will use the following Chernoff-type bound for a sum of
independent Bernoulli random variables.\COMMENT{formally, we do need it in this generality
as when adding edges of $H_1$, some of the $p_i$ may be 0.}

\begin{theorem}[Chernoff Inequality]\label{Chernoff}
Let $X_1,\ldots,X_n$ be independent Bernoulli random variables with
parameters $p_1,\ldots,p_n$ respectively and let $X = X_1 + \cdots +
X_n$. Then
\[
\P(|X - \E X| \geq t) \leq 2\exp{\left(- \frac{t^2}{3 \E X}
\right)}.
\]
\end{theorem}

In particular, since a binomial random variable $X$ with parameters
$n$ and $p$ is a sum of $n$ independent Bernoulli random variables,
the above inequality holds for binomial random variables as well.

\subsection{Regularity Lemma}

In the proof, we will use the degree form of
\Szemeredi's Regularity Lemma. Before stating it, we need to
introduce some notation. The {\em density} of a bipartite graph $G =
(A,B)$ with vertex classes $A$ and $B$ is defined to be $d_G(A,B) :=
\frac{e(A,B)}{|A||B|}$. We sometimes write $d(A,B)$ for $d_G(A,B)$ if this is
unambiguous. Given $\eps > 0$, we say that $G$ is
$\eps$-{\em regular} if for all subsets $X \subseteq A$ and $Y
\subseteq B$ with $|X| \geq \eps|A|$ and $|Y| \geq \eps |B|$ we have
that $|d(X,Y) - d(A,B)| < \eps$. Given $d\in [0,1]$,
we say that $G$ is $(\eps,d)$-{\em super-regular}
if it is $\eps$-regular and furthermore $d_G(a) \geq d|B|$ for all
$a \in A$ and $d_G(b) \geq d|A|$ for all $b \in B$. We will use the
following degree form of \Szemeredi's Regularity Lemma:

\begin{lemma}[Regularity Lemma; Degree form]
For every $\eps \in (0,1)$ and each positive integer $M'$, there are
positive integers $M$ and $n_0$ such that if $G$ is any graph on $n
\geq n_0$ vertices and $d \in [0,1]$ is any real number, then there
is a partition of the vertices of $G$ into $k+1$ classes
$V_0,V_1,\ldots,V_k$, and a spanning subgraph $G'$ of $G$ with the
following properties:
\begin{itemize}
\item $M' \leq k \leq M$;
\item $|V_0| \leq \eps n, |V_1| = \cdots = |V_k| =:m$;
\item $d_{G'}(v) \geq d_G(v) - (d + \eps)n$ for every $v \in V(G)$;
\item $G'[V_i]$ is empty for every $0 \leq i \leq k$;
\item all pairs $(V_i,V_j)$ with $1 \leq i < j \leq k$ are
$\eps$-regular with density either 0 or at least~$d$.
\end{itemize}
\end{lemma}

We call $V_1,\ldots,V_k$ the {\em clusters} of the partition and
$V_0$ the {\em exceptional set}.

\subsection{Factor Theorems}
An $r$-{\em factor} of a multigraph $G$ is an $r$-regular
submultigraph $H$ of $G$. We will use the following classical result of Petersen.
\begin{theorem}[Petersen's Theorem]\label{Petersen}
Every regular multigraph of positive even degree contains a
2-factor.
\end{theorem}

Furthermore, we will use Tutte's $f$-factor theorem~\cite{Tutte98} which gives a necessary and sufficient condition for a
multigraph to contain an $f$-factor. (In fact, the theorem is more
general.) Before stating it we need to introduce some notation. Given a
multigraph $G$, a positive integer $r$, and disjoint subsets $T,U$
of $V(G)$, we say that a component $C$ of $G[U]$ is {\em odd (with
respect to $r$ and $T$)} if $e(C,T) + r|C|$ is odd. We write $q(U)$
for the number of odd components of~$U$.

\begin{theorem}\label{Tutte}
A multigraph $G$ contains an $r$-factor if and only if for every
partition of the vertex set of $G$ into sets $S,T,U$, we have
\begin{equation}\label{Tutte1}
\sum_{v \in T}d(v) - e(S,T) + r(|S| - |T|) \geq q(U).
\end{equation}
\end{theorem}

In fact, we will only need the following consequence of
Theorem~\ref{Tutte}.

\begin{theorem}\label{Factor Theorem}
Let $G$ be a multigraph on $n$ vertices of minimum degree $\delta
\geq \ell n/2$, in which every pair of vertices is joined by at most
$\ell$ edges.
\begin{itemize}
\item[(i)] Let $r$ be an even number such that $r \leq \frac{\delta +
\sqrt{\ell n (2 \delta - \ell n)}}{2}$. Then $G$ contains an
$r$-factor.
\item[(ii)] Let $0 < \xi < 1/9$ and suppose
$ (1/2 + \xi)\ell n \leq \Delta(G) \leq \delta + \xi^2 \ell n$. If $r$ is an even number
such that $r \leq \delta - \xi \ell n$ and $n$ is sufficiently
large, then $G$ contains an $r$-factor.
\end{itemize}
\end{theorem}
The case when $\ell=1$ and $r$ is close to $n/4$ in~(i) was already proven by
Katerinis~\cite{katerinis}.

\begin{proof}
By Theorem~\ref{Petersen}, in both (i) and (ii) it suffices to consider the case that
$r$ is the maximal positive even integer satisfying the conditions.
Observe that since $\delta \leq \ell (n-1) < \ell n$ it follows that
$\ell n(2\delta - \ell n) = \delta^2 - (\ell n - \delta)^2 <
\delta^2$, so in case (i) we have $r < \delta$ and since both $r$
and $\delta$ are integers we have $r \leq \delta - 1$.\COMMENT{Here
I am using the fact that the multigraph is loopless.} This also
holds in case (ii).

By Theorem~\ref{Tutte}, it is enough to show (in both cases) that
\eqref{Tutte1} holds for every partition of the vertex set of $G$
into sets $S,T$ and $U$.

\medskip

\noindent
\textbf{Case~1.} \emph{$\ell|T| \leq r-1$ and $\ell |S|
\leq \delta - r$.}

\smallskip

Since in this case $d_T(v) \leq \ell |T| \leq r-1$ for every $v \in
V(G)$, the left hand side of \eqref{Tutte1} is
\[
\sum_{v \in T}(d(v) - r) + \sum_{v \in S}(r - d_T(v)) \geq |T| +
|S|.
\]
So in this case, it is enough to show that $q(U) \leq |T| + |S|$. If
$|T| = 0$, the result holds since in this case no component of $G[U]$
is odd, i.e~$q(U)=0$. If $|T|=1$ and $|S|=0$, then the degree
conditions imply that $G[U]$ is connected and so $q(U) \leq 1 = |T| +
|S|$. (Indeed, the degree conditions imply that the undirected graph
obtained from $G$ by ignoring multiple edges has minimum degree at
least $n/2$ and so any subgraph of it on $n-1$ vertices must be
connected.) Thus in this case, we may assume that $2 \leq |T| + |S|
\leq \frac{\delta - 1}{\ell}$. Observe that every vertex $v\in U$
has at most $\ell(|T| + |S|)$ neighbours in $T \cup S$ when counting
multiplicity and so it has at least $\frac{\delta - \ell(|T| +
|S|)}{\ell}$ distinct neighbours in $U$. In particular, every
component of $G[U]$ contains at least $\frac{\delta + \ell - \ell(|T| +
|S|)}{\ell}$ vertices and so certainly $q(U) \leq \frac{\ell
|U|}{\delta + \ell - \ell(|T| + |S|)}$. Writing $k:= |T| + |S|$, it
is enough in this case to prove that $k \geq \frac{\ell(n-k)}{\delta
+ \ell - \ell k}$. But this is equivalent to proving that $k\delta +
2k\ell - \ell k^2 - \ell n \geq 0$, which is true since the left hand
side is equal to $(k-2)(\delta - \ell k) + 2\delta - \ell n$.

\medskip

\noindent \textbf{Case~2.} $\ell |T| = r$.

\smallskip

Since for every vertex $v\in S$ we have $d_T(v) \leq \ell|T| = r$, it
follows that $r|S|\ge e(S,T)$. Thus the left hand side of \eqref{Tutte1} is at least
$(\delta - r)|T| = (\delta - r)r/\ell$. Observe that $G[U]$ has at most
$n - \delta/\ell$ components. Indeed, if $C$ is a component of $G[U]$
and $x$ is a vertex of $C$, then as $x$ can only have neighbours in
$C \cup S \cup T$ we have that $|C \cup S \cup T| \geq 1 +
\delta/\ell$ and so $U$ has at most $n - 1 - \delta/\ell$ other
components. Thus, it is enough to show that $(\delta - r)r \geq \ell
n - \delta$. For case (ii) (recall that $r$ is maximal subject to
the given conditions) we have $(\delta - r)r \geq \xi \ell n(\delta
- \xi \ell n - 2) \ge \ell n - \delta$. To see the last inequality, recall that $\delta \ge \ell n/2$ and $\xi \ell n\ge 4$ say (as we assume that $n$ is sufficiently large). So $\xi \ell n(\delta-\xi \ell n-2)\ge 4(\ell n(1/2-\xi)-2) = (2 - 4\xi) \ell n - 8 \ge \ell n$. To prove (i), note that we (always) have
\begin{equation}\label{Tutte2}
(\delta - r)r = \frac{\delta^2}{4} - \left(r - \frac{\delta}{2}
\right)^2 \geq  \frac{\delta^2}{4} - \frac{\ell n(2\delta - \ell
n)}{4} = \left(\frac{\ell n - \delta}{2}\right)^2.
\end{equation}
So the result also holds in case (i) unless $\delta \geq \ell n -
3$. But if this is the case, then $(\delta - r)r \geq \ell n -
\delta$ unless%
   \COMMENT{Enough to show that $(\delta-r)r\ge 3$. But $r\le \delta-1$ and
$r$ is even, so this holds unless $r=2$ and $\delta=3$. But then
$2=(\delta-r)r\ge \ell n-\delta=\ell n-3$ unless $\ell n\ge 6$,
in which case $\ell n=6$ as $\delta\ge \ell n/2$.}
$r=2,\delta=3$ and $\ell n = 6$. But this violates
the assumption on $r$ in~(i).

\medskip

\noindent \textbf{Case~3.} \emph{$|T| \geq \frac{r +
1}{\ell}$ and $|S| \geq \frac{\delta - r + 1}{\ell}$.}

\smallskip

Since $q(U) \leq |U| =  n - |S| - |T|$, it is enough to show that in
this case we have
\begin{equation}\label{Tutte3}
(\delta - r + 1)|T| + (r + 1)|S| - \ell |S||T| \geq n.
\end{equation}
By writing the left hand side of~\eqref{Tutte3} as
\begin{equation}\label{Tutte3a}
\frac{(\delta - r + 1)(r+1)}{\ell} - \ell \left(|T| -
\frac{r+1}{\ell} \right)\left(|S| - \frac{\delta - r + 1}{\ell}
\right),
\end{equation}
we observe that it is minimized when $|T| + |S|$ is maximal,%
   \COMMENT{Here we need that both $|T| \geq \frac{r +
1}{\ell}$ and $|S| \geq \frac{\delta - r + 1}{\ell}$ and not just one of them}
i.e.~it is equal to $n$. To prove (i), observe that the left hand side
of~\eqref{Tutte3} is at least%
   \COMMENT{Since $xy\le ((x+y)/2)^2$ for all $x,y$, we have that
$(t-a)(s-b)\le (\frac{t+s-a-b}{2})^2\le (\frac{n-a-b}{2})^2$ where $t=|T|$, $s=|S|$ etc}
\begin{align}\label{Tutte3b}
&\frac{(\delta - r + 1)(r+1)}{\ell} - \frac{\ell}{4}\left(n - \frac{\delta + 2}{\ell} \right)^2
= \frac{(\delta - r)r}{\ell} + \frac{\delta + 1}{\ell} -
\frac{\ell n^2}{4} + \frac{n (\delta + 2)}{2} - \frac{(\delta + 2)^2}{4 \ell}\nonumber \\
&\stackrel{(\ref{Tutte2})}{\geq} \frac{(\ell n - \delta)^2}{4 \ell}
+ \frac{\delta + 1}{\ell} - \frac{\ell n^2}{4} + \frac{n (\delta +
2)}{2} - \frac{(\delta + 2)^2}{4 \ell} = n.
\end{align}
To prove (ii) we may assume that $\delta < (1 - \sqrt{\xi}) \ell
n$. Indeed if $\delta \geq (1 - \sqrt{\xi}) \ell n$, then using that $r$ is maximal
subject to the given conditions we have
\[
(\delta - r)r \geq \xi \ell n(\delta - \xi \ell n - 2) \geq \xi
\ell n \left( \left( \frac{1}{2} - \xi \right)\ell n -2 \right) \ge
\frac{\xi \ell^2 n^2}{4} \geq \frac{(\ell n - \delta)^2}{4}
\]
and the result follows exactly as in case (i).

If also $|T| \leq \frac{\Delta}{\ell}$, then we
claim that $|T| - \frac{r+1}{\ell} \leq |S| - \frac{\delta - r +
1}{\ell}$. Indeed, this follows since%
   \COMMENT{Use that $r$ is maximal subject to the given conditions to see
the 1st term in the 2nd line.}
\begin{multline*}
|T| - |S| \leq \frac{2\Delta}{\ell} - n \leq \frac{2\delta}{\ell}+
2\xi^2 n - n \leq \frac{\delta}{\ell} + 2\xi^2 n - \sqrt{\xi}n \leq
\\
\frac{(2r - \delta) + 2\xi \ell n + 4}{\ell} + 2\xi^2n -
\sqrt{\xi}n = \frac{2r - \delta}{\ell} + (2\xi^2 + 2\xi -
\sqrt{\xi})n + \frac{4}{\ell} \le \frac{2r - \delta}{\ell}.
\end{multline*}
This claim together with the fact that $|T|+|S|=n$
implies that~\eqref{Tutte3a} (and thus the left hand side of~\eqref{Tutte3}) is
minimized when%
    \COMMENT{Set $x=|T|-\frac{r+1}{\ell}$ and $c=n-\frac{\delta  +
2}{\ell}$. Then $(|T| - \frac{r+1}{\ell})( |S| - \frac{\delta - r +
1}{\ell})=x(c-x)$. But the derivative of $x(c-x)$ is $c-2x=(c-x)-x$
which is at least $0$ since $x\le c-x$. So $(|T| - \frac{r+1}{\ell})( |S| - \frac{\delta - r +
1}{\ell})$ is maximized when $|T| = \Delta/\ell$.}
$|T| = \Delta/\ell$ and $|S| = n - \Delta/\ell$.
Note that $|T| \geq (1/2 + \xi)n$ in this case and so $|T|-|S|\ge 2\xi n$.
Thus the left hand side of~\eqref{Tutte3} is at least
\[
(\delta - r)|T| + (r - \ell|T|)|S| = (\delta - r)(|T| - |S|) +
(\delta - \Delta)|S| \\
\geq 2 \xi^2 \ell n^2 - \xi^2 \ell n^2 \ge n.
\]
To complete the
proof, suppose $|T| \geq \frac{\Delta}{\ell}$. Then $e(S,T) \leq
\Delta|S|$ and again $|T|-|S|\ge 2\xi n$. So the left hand side of~\eqref{Tutte1} is at least
\[
(\delta - r)|T| + (r - \Delta)|S| = (\delta - r)(|T| - |S|) +
(\delta - \Delta)|S| \geq 2 \xi^2 \ell n^2 - \xi^2 \ell n^2 \ge n.
\]

\medskip

\noindent \textbf{Case~4.} \emph{$|T| \geq \frac{r +
1}{\ell}$ or $|S| \geq \frac{\delta - r + 1}{\ell}$ but not both.}

\smallskip

As in Case~3 it suffices to show that~\eqref{Tutte3} holds. \eqref{Tutte3a}
shows that in this case the left hand side of~\eqref{Tutte3} is at least
$(\delta - r + 1)(r+1)/\ell$. So~(i) holds since~\eqref{Tutte3b} implies that
the left hand side of~\eqref{Tutte3} is at least~$n$. For~(ii), note that
$$\frac{(\delta - r + 1)(r+1)}{\ell}\ge \frac{\xi \ell n(r+1)}{\ell}
\ge \xi n(\delta-\xi \ell n-1)\ge n.$$
(Here we use the maximality of $r$ in both inequalities.)
\end{proof}

%%%%%%%%%%%%%%%%%%%%%%%%%%%%%%%%%%%%%%%%%%%%%%%%%%%%%%%%%%%%%%%%%%%%%%%%%%%%%%%%%%%%%%%%%%%%%%%%%%%%%%%%%%%%%%%%%%%%%%%%%%%%%%%%%%%%%%%%%%%%%%%%%%%%%%%%%
\section{Proofs of the Main Theorems}

In this section we will prove Theorems~\ref{Almost regular}
and~\ref{Minimum Degree}(ii) simultaneously.
Observe that in both cases we may assume that $\alpha \ll 1$.
Define additional constants such that
\[ \frac{1}{n_0} \ll \zeta \ll 1/M'\ll \eps \ll \beta \ll \eta \ll d \ll \gamma \ll \alpha \]
and let $G$ be a graph
on $n\ge n_0$ vertices with minimum degree $\delta \geq (1/2 + \alpha)n$
and maximum degree $\Delta$.

\subsection{Applying the Regularity Lemma}

We apply the Regularity Lemma to $G$ with parameters $\eps/2$, $3d/2$
and $M'$ to obtain a partition of $G$ into clusters
$V_1,\ldots,V_k$ and an exceptional set $V_0$, and a spanning subgraph $G'$ of $G$.
Let $R$ be the
multigraph on vertex set $\{V_1,\ldots,V_k\}$ obtained by adding
exactly $\ell_{ij} := \lfloor d_{G'}(V_i,V_j)/\beta \rfloor$ multiple
edges between the vertices $V_i$ and $V_j$ of $R$. By removing one vertex from
each cluster if necessary and adding all these vertices to~$V_0$,
we may assume that $m:=|V_1|=\dots=|V_k|$ is even. So now $|V_0|\leq \eps n/2+k\leq \eps n$.
The next lemma shows that $R$ inherits its minimum and maximum
degree from~$G$.

\begin{lemma}\label{Degrees of R} $ $
\begin{itemize}
\item[(i)] $\delta(R) \geq \left( \frac{\delta}{n} - 2d \right)\frac{k}{\beta}$;
\item[(ii)] $\Delta(R) \leq \left( \frac{\Delta}{n} + 2d \right)\frac{k}{\beta}$.
\end{itemize}
\end{lemma}

\begin{proof}
For any cluster $V_i$ of $R$ we have
\[
\sum_{x \in V_i}d_{G'}(x) \leq e(V_0,V_i) + \sum_{j\neq i}
e_{G'}(V_i,V_j) \leq \eps mn + \sum_{j\neq i} d_{G'}(V_i,V_j) m^2.
\]
Since $d_{G'}(V_i,V_j) \leq \beta(\ell_{ij} + 1)$, we obtain
\[
\sum_{x \in V_i}d_{G'}(x) \leq \eps mn + (d_R(V_i) + k)\beta m^2.
\]
By the definition of $G'$ in the Regularity Lemma we also have
\[
\sum_{x \in V_i}d_{G'}(x) \geq \sum_{x \in V_i}\left( d_G(x) - (3d/2 + \eps)n\right)
\geq \delta m - (3d/2 + \eps)mn.
\]
Since also $\eps,\beta \ll d$, (i) follows. Similarly,
\[
d_R(V_i) \beta m^2 \leq \sum_{j\neq i} d_{G'}(V_i,V_j)m^2 \leq \sum_{x \in
V_i}d_{G'}(x) \leq \Delta m,
\]
so (ii) follows.
\end{proof}
Since $\delta \geq (1/2 + \alpha)n$ and since between any two
vertices of $R$ there are at most $1/\beta$ edges,
Theorem~\ref{Factor Theorem}(i) implies that $R$ contains an
$r$-regular submultigraph $R'$ for every even positive integer $r$
satisfying
\[
r \leq \left(\frac{\delta}{n} -2d  +
\sqrt{\frac{2\delta}{n} - 4d - 1}   \right) \frac{k}{2\beta} =
\left( \delta - 2dn  + \sqrt{2\delta n - 4dn^2 - n^2} \right)
\frac{k}{2\beta n}.
\]
In particular, (using the inequality $\sqrt{x-y} \geq \sqrt{x} -
\sqrt{y}$ for $x\ge y > 0$ and the fact that $\alpha \gg d$) we may
assume that
\begin{equation} \label{mindegr}
r = \left(\delta + \sqrt{n(2\delta - n)} - \alpha n/2
\right) \frac{k}{2 \beta n}.
\end{equation}
Moreover, for the proof of Theorem~\ref{Almost regular}, we have
$\Delta(R) - \delta(R) \leq \left(\frac{\Delta - \delta}{n} +
4d\right) \frac{k}{\beta} \leq \alpha^2 k/4\beta$. Therefore, by taking
$\xi = \alpha/2$ in Theorem~\ref{Factor Theorem}(ii) we may even
assume that
\begin{equation} \label{regularr}
r = \left(\delta - 2\alpha n/3 \right)\frac{k}{\beta n}.
\end{equation}
So from now on, $R'$ is an $r$-regular submultigraph of $R$, where
$r$ is even and is given by~(\ref{mindegr}) for the proof of
Theorem~\ref{Minimum Degree}(ii) and given by (\ref{regularr}) for the proof of
Theorem~\ref{Almost regular}.

By Theorem~\ref{Petersen}, $R'$ can be decomposed into 2-factors. As
mentioned in the overview, it will be more convenient to work with a
matching decomposition rather than a 2-factor decomposition. If all
the cycles in all the $2$-factor-decompositions had even length then we could
decompose them into matchings. Because this might not be the case,
we will split each cluster corresponding to a vertex of $R$ into two
clusters to obtain a new multigraph $R^*$. More specifically, for
each $1 \leq i \leq k$, we split each cluster $V_i$ arbitrarily into
two pieces $V_i^1$ and $V_i^2$ of size $m/2$. $R^*$ is defined to be the
multigraph on vertex set $V_1^1,V_1^2,\ldots,V_k^1,V_k^2$ where the
number of multiedges between $V_i^a$ and $V_j^b$ ($1 \leq i,j \leq
k, 1 \leq a,b \leq 2$) is equal to the number of multiedges of $R$
between $V_i$ and $V_j$.

Recall that by Theorem~\ref{Petersen}, $R'$ can be decomposed into
2-factors. We claim that each cycle $v_1 \ldots v_t$ of each
2-factor gives rise to two edge-disjoint even cycles in $R^*$ each
of length $2t$, which themselves give rise to a total of four
matchings in $R^*$, each of size $t$. Indeed, denoting by $a_i$ and
$b_i$ the clusters in $R^*$ corresponding to $v_i$, if $t$ is even,
say $t = 2s$, then we can take the cycles $a_1a_2 \ldots a_{2s} b_1
b_2 \ldots b_{2s}$ and $a_1b_2 \ldots a_{2s-1}b_{2s}b_1a_2 \ldots
b_{2s-1}a_{2s}$. If $t$ is odd, say $t = 2s+1$, then we can take the
cycles $a_1b_2 \ldots a_{2s-1}b_{2s}a_{2s+1}b_1b_{2s+1}a_{2s} \ldots
b_3a_2$ and $a_1 a_{2s+1} a_{2s} \ldots a_2 b_1 b_2 \ldots
b_{2s+1}$ (see Figure~\ref{cycles} for the cases $t=4,5$).

\begin{figure}[h]
\includegraphics[scale=0.08]{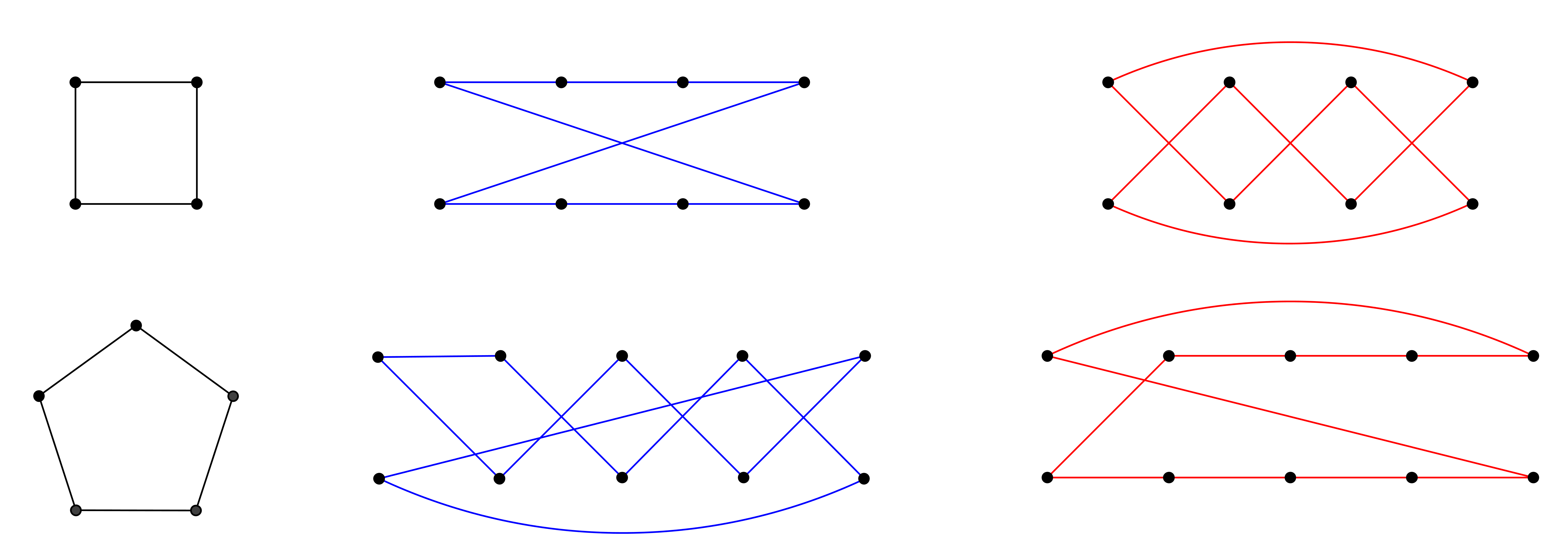}
\caption{Cycles in $R'$ and the corresponding cycles in $R^*$.}
\label{cycles}
\end{figure}

To simplify the notation we will now make the following relabelings: $R'$ has served its purpose in finding a set of edge-disjoint perfect matchings in $R^*$ and it will not be used any more, $R^*$ is relabelled to $R$ and the clusters $V_1^1,V_1^2,\ldots,V_k^1,V_k^2$ are relabelled to $V_1,\ldots,V_{k'}$. We also relabel $k'$ back to $k$. Note that now each $V_i$ has size $m' = m/2$ but we relabel $m'$ back to $m$. 

In particular we can now assume that we have a partition of the vertex set of $G$ into $k$ clusters
$V_1,\ldots,V_k$ and an exceptional set $V_0$, and a spanning subgraph $G'$ of $G$ satisfying the
following properties:

\begin{itemize}
\item $|V_0| \leq \eps n$ and $|V_1| = \cdots = |V_k| =: m$;
\item $d_{G'}(v) \geq d_G(v) - (3d/2 + \eps)n$ for every $v \in V(G)$;
\item $G'[V_i]$ is empty for every $0 \leq i \leq k$;
\item all pairs $(V_i,V_j)_{G'}$ with $1 \leq i < j \leq k$ are
$\eps$-regular with density either 0 or at least~$d$;
\item $R$ is a multigraph on vertex set $V_1,\ldots,V_k$ having
exactly $\ell_{ij} = \lfloor \frac{d_{G'}(V_i,V_j)\pm \eps}{\beta} \rfloor$ edges
$f_{ij}^1, \ldots, f_{ij}^{\ell_{ij}}$ joining $V_i$ and $V_j$;
\item $R$ has minimum degree at least $(\delta - 2dn)
\frac{k}{\beta n}$ and maximum degree at most $(\Delta + 2dn)
\frac{k}{\beta n}$;
\item $R$ contains a set of $r$ edge-disjoint perfect matchings,
where $r$ satisfies~\eqref{regularr} for
Theorem~\ref{Almost regular} and \eqref{mindegr} for
Theorem~\ref{Minimum Degree}(ii).
\end{itemize}
Later on, we will use  that in both cases we have%
    \COMMENT{Bound on $\delta$ also true for Theorem~\ref{Minimum Degree} since
$\sqrt{n(2\delta -n)}<\delta$. Indeed, need that $2\delta n-n^2\le \delta^2$,
ie $\delta^2- 2\delta n+n^2\ge 0$. The LHS is minimized when $\delta=n$, in which
case the LHS equals $0$.}
\begin{equation} \label{lowerr}
k/5\beta\le r \le k/\beta \qquad \mbox{and} \qquad \delta \ge r\beta m +\alpha n/5.
\end{equation}
We let $M_1,\ldots,M_r$ be $r$ edge-disjoint perfect matchings of $R$.
We will define edge-disjoint subgraphs $G_1,\ldots,G_r$ of $G$ corresponding
to the matchings $M_1,\ldots,M_r$. Before doing that, for each $1
\leq i < j \leq k$ we will find $\ell_{ij}$ disjoint subsets $E_{ij}^1, \ldots,
E_{ij}^{\ell_{ij}}$ of $E_{G'}(V_i,V_j)$ corresponding to the $\ell_{ij}$ edges
$f_{ij}^1, \ldots, f_{ij}^{\ell_{ij}}$ of $R$ between $V_i$ and
$V_j$. The next well known observation shows that we can choose the
$E_{ij}^{\ell}$ so that each $(V_i,V_j)_{E_{ij}^{\ell}}$ forms a
regular pair. It is e.g.~a special case of Lemma~10(i)
in~\cite{Kuhn&Osthus&Treglown}. To prove it, one considers a random partition of the
edges of $G'$ between $V_i$ and~$V_j$.

\begin{lemma}\label{Multiedges correspond to regular pairs}
For each $1 \leq i < j \leq k$, there are $\ell_{ij}$ edge-disjoint subsets $E_{ij}^1, \ldots,
E_{ij}^{\ell_{ij}}$ of $E_{G'}(V_i,V_j)$ such that each $(V_i,V_j)_{E_{ij}^{\ell}}$ is
$\eps$-regular of density either 0 or $\beta \pm \eps$.%
\COMMENT{was not true with `partition' instead of `subsets' due to rounding effects}
\end{lemma}

Given a matching $M_i$, we define the graph $G_i$ on vertex set
$V(G)$ as follows: Initially, the edge set of $G_i$ is the union of
the sets $E_{ab}^{\ell}$, taken over all edges $f_{ab}^{\ell}$
of $M_i$. So at the moment, $G_i$ is a disjoint union of $V_0$ and
$k':=k/2$ pairs which are $\eps$-regular and have density $\beta
\pm \eps$. For every such pair, by removing exactly $2\eps m$
vertices from each cluster of the pair, we may assume that the pair
is $2\eps$-regular and that every vertex remaining in each
cluster has degree $(\beta \pm 4\eps) m$ within the pair. (In
particular, it is $(2\eps,\beta-4\eps)$-super-regular.) We denote
by $V_{0i}$ the union of $V_0$, together with the set of all these
removed vertices. Observe that
\begin{equation} \label{v0i}
|V_{0i}| \leq \eps n + 2\eps mk
\leq 3\eps n.
\end{equation}
Finally, we remove all edges incident to vertices of
$V_{0i}$. We will denote the pairs of clusters of $G_i$ corresponding to the
edges of~$M_i$ by $(U_{1,i},V_{1,i}), \ldots , (U_{k',i},V_{k',i})$ and call
them the \emph{pairs of clusters of~$G_i$}. Observe that
every cluster $V$ of $G_i$ is contained in a unique cluster of
$R$, which we will denote by $V^R$, and each cluster $V$ of $R$ contains a unique cluster of $G_i$, which we will denote by $V(i)$. In particular we have that $|V \setminus V(i)| \leq 2\eps m$.\COMMENT{So $V^R(i) = V$ for every $V$ in $G_i$ and $V(i)^R =
V$ for every $V$ in $R$ and every $1 \leq i \leq r$.}

So we have exactly $r$ edge-disjoint spanning subgraphs $G_i$ of $G$
such that for each $1 \leq i \leq r$ the following hold:
\begin{itemize}
\item [($a_1$)] $G_i$ is a disjoint union of a
set $V_{0i}$ of size at most $3 \eps n$ together with $k$ clusters
$U_{1,i},V_{1,i}, \ldots , U_{k',i},V_{k',i}$ each of size exactly
$(1 - 2\eps) m$;
\item [($a_2$)] For each $x \in V(G)$ the
degree of $x$ in $G_i$ is either~0 if $x \in V_{0i}$ or $(\beta \pm
4 \eps) m$ otherwise;
\item [($a_3$)] For each $1 \leq j \leq k'$
the pair $(U_{j,i},V_{j,i})$ is $(2\eps,\beta -4\eps)$-super-regular;
\item [($a_4$)] Every edge of $G_i$ lies in one of the pairs
$(U_{j,i},V_{j,i})$ for some $1 \leq j \leq k'$.
\end{itemize}

\subsection{Extracting random subgraphs}

At the moment, no $G_i$ contains a Hamilton cycle. Our aim is to add
some of the edges of $G$ which do not belong to any of the $G_i$
into the $G_i$ in such a way that the graphs obtained from the $G_i$
are still edge-disjoint and
each of them contains almost $\beta m/2$ edge-disjoint Hamilton cycles.
To achieve this it will be convenient however to remove some of the
edges of each $G_i$ first while still keeping most of its properties.

We will show that there are edge-disjoint subgraphs $H_1,H_2$ and $H_3$
of $G$ satisfying the following properties:

\begin{lemma}\label{Properties after definition of H's}
There are edge-disjoint subgraphs $H_1,H_2$ and $H_3$ of $G$ such that the
following properties hold:
\begin{itemize}
\item[(i)] For every vertex $x$ of $G$ and every $1 \leq
j \leq 3$ we have $|d_{H_j}(x) - \gamma d_G(x)| \leq \zeta n$.
\item[(ii)] For every vertex $x$ of $G$, every $1 \leq i \leq r$ and every $1 \leq
j \leq 3$
\[
\left|d_{H_j\cap G_i}(x) - \gamma d_{G_i}(x) \right| \leq \zeta
n.\COMMENT{At the moment, we have not yet removed the edges of
$H_j$'s from the $G_i$'s.}
\]
\item[(iii)] For every vertex $x$ of $G$, every $1 \leq i \leq r$ and every $1 \leq
j \leq 3$
\[
\left| \left| N_{H_j}(x) \cap V_{0i} \right| - \gamma \left|
N_{G}(x) \cap V_{0i} \right| \right| \leq \zeta n.
\]
\item[(iv)] For every vertex $x$ of $G$, every $1 \leq i \leq r$,
every $1 \leq t \leq k$ and every $1 \leq j \leq 3$
\[
\left| \left| N_{H_j\cap G_i}(x) \cap V_t \right| - \gamma \left| N_{G_i}(x)
\cap V_t \right| \right| \leq \zeta n.
\]
\item[(v)] For every vertex $x$ of $G$, every $1 \leq t \leq k$ and
every $1 \leq j \leq 3$
\[
\left| \left| N_{H_j}(x) \cap V_t \right| - \gamma \left| N_{G}(x)
\cap V_t \right| \right| \leq \zeta n.
\]
\item[(vi)] For every $1 \leq i \leq r$, every pair of clusters
$(U,V)$ of $G_i$, every $A \subseteq U$ and every $B \subseteq V$ with $|A|,|B|\ge 2\eps |U|$
and every $1 \leq j \leq 3$ we have
\[
\left| \left|E_{H_j \cap G_i}(A,B) \right|  - \gamma \left|
E_{G_i}(A,B) \right| \right| \leq \zeta n^2.
\]
\item[(vii)] For all clusters
$U\neq V$ of $R$, every $A \subseteq U$and every $B \subseteq V$ with $|A|,|B|\ge \eps m$
and every $1 \leq j \leq 3$ we have
\[
\left| \left|E_{H_j \cap G'}(A,B) \right|  - \gamma \left|
E_{G'}(A,B) \right| \right| \leq \zeta n^2.
\]
\end{itemize}
\end{lemma}

\begin{proof}
We construct the $H_j$'s randomly as follows: For every edge $e$ of
$G$, with probability $3\gamma$, we assign it uniformly to one of
the $H_j$'s and with probability $1 - 3\gamma$ to none of them.
By Theorem~\ref{Chernoff}, all properties hold with high
probability. More specifically, the total probability of failure is
at most
\[
(6n + 6rn +6rn + 6rkn+6kn) \exp{\left(- \frac{\zeta^2 n}{3\gamma} \right)} +
(3rk4^m+3k^2 4^m) \exp{\left(- \frac{\zeta^2 n^2}{3\gamma} \right)} \ll 1.
\qedhere
\]
\end{proof}
We pick subgraphs $H_1,H_2$ and $H_3$ of $G$ as given by Lemma~\ref{Properties after
definition of H's}.
It will be convenient for later use to split (a subgraph of) $H_3$ into $r$
subgraphs called $H_{3,1}, \ldots , H_{3,r}$ satisfying the
properties of the following lemma. For each $i$, we will add
edges of $H_{3,i}$ to $G_i$ (but not to any of the other $G_j$)
during the final part of our proof (see Section~\ref{sec:finalpf}).
Roughly speaking, if $(U,V)$ is an edge of $R$, then we require
$H_{3,i}$ to contain some edges between $U$ and $V$ (but we
do not need many of these edges).
If $(U,V)$ corresponds to a matching edge of $M_i$, then we also
require the corresponding subgraph of $H_{3,i}$ to be reasonably dense.
Moreover, each edge of $H_{3,i}$ will correspond to some edge of~$R$.

\begin{lemma}\label{Distribution of edges of H_3 within the G_i's}
There are edge-disjoint subgraphs $H_{3,1}, \ldots, H_{3,r}$ of $H_3$
so that the following hold:
\begin{itemize}
\item[(i)] For every $1 \leq i \leq r$, all clusters
$U\neq V$ of $G_i$ such that $U^R$ and $V^R$ are adjacent in $R$ and
every $U' \subseteq U$ and $V' \subseteq V$ with $|U'|,|V'| \geq
\eps m$ there are at least $\frac{\gamma \beta \eps^2 d m^2}{5k}$
edges between $U'$ and $V'$ in $H_{3,i}$;
\item[(ii)] For every $1 \leq i \leq r$ and every $1 \leq j \leq k'$, the
pair $(U_{j,i},V_{j,i})_{H_{3,i}}$ is $(5\eps/2,\gamma \beta /5)$-super-regular;
\item[(iii)] For every $1 \leq i \leq r$, $H_{3,i}$ has maximum degree
at most $\beta m$;\COMMENT{this is needed in the proof of the decomposition lemma}
\item[(iv)] For every $1 \leq i \leq r$ and every edge $e$ of $H_{3,i}$ there
are clusters $U\neq V$ of $G_i$ such that such that $U^R$ and $V^R$ are adjacent in $R$
and $e$ joins $U$ to~$V$.
\end{itemize}
\end{lemma}

\begin{proof}
Recall that given any two adjacent vertices $V_a,V_b$ of $R$, and
any $1 \leq \ell \leq \ell_{ab}$, there is at most one $M_i$ which
contains the edge $f_{ab}^{\ell}$. If there is no such $M_i$, then
we assign the edges of $E_{ab}^{\ell} \cap E(H_3)$ to the
$H_{3,j}$ uniformly and independently at random. If there is such
an $M_i$, we assign every edge of $E_{ab}^{\ell}\cap E(H_3)$ to $H_{3,i}$ with
probability $1/2$ or to one of the other $H_{3,j}$'s uniformly at
random. Note that this means that every edge of $H_3$ between $V_a$ and $V_b$ which lies
in some $G_i$ is assigned to $H_{3,i}$ with probability $1/2$ and assigned to
some other $H_{3,j}$ with probability $1/2(r-1)$.

To prove (i), observe that since $r \leq k/\beta$ by~(\ref{lowerr}), every edge of
$H_3$ with endpoints in $U$ and $V$ has probability at least
$\beta/2k$ of being assigned to $H_{3,i}$. Since $(U^R,V^R)_{G'}$ is
$\eps$-regular of density at least $d$, there are at least
$\eps^2 dm^2$ edges between $U'$ and $V'$ in $G'$ and so by
Lemma~\ref{Properties after definition of H's}(vii), $H_3$
contains at least $\gamma \eps^2 d m^2/2$ such edges. So by
Theorem~\ref{Chernoff}, (i) holds with high probability.

To prove (ii), recall that before defining $H_3$, the pair
$(U_{j,i},V_{j,i})_{G_i}$ was $(2\eps,\beta-4\eps)$-super-regular by~$(a_3)$.
Thus by Lemma~\ref{Properties after
definition of H's}(iv) and~(vi), $(U_{j,i},V_{j,i})_{H_3\cap G_i}$ is $(2\eps, \gamma
\beta /2)$-super-regular. Since every edge of
$(U_{j,i},V_{j,i})_{H_3\cap G_i}$ has probability exactly 1/2 of being
assigned to $H_{3,i}$, another application of Theorem~\ref{Chernoff}
shows that with high probability $(U_{j,i},V_{j,i})_{H_{3,i}\cap G_i}$ is $(2\eps, \gamma
\beta /5)$-super-regular. On the other hand, for every edge $e$ in $E(H_3)\setminus E(G_i)$ between
$U_{j,i}$ and $V_{j,i}$ the probability that $e$ is assigned to $H_{3,i}$
is at most $1/r\le 5\beta/k\ll \eps$ (the first inequality follows from~(\ref{lowerr})).
Together with Theorem~\ref{Chernoff} this
implies that with high probability $(U_{j,i},V_{j,i})_{H_{3,i}}$ consists of
$(U_{j,i},V_{j,i})_{H_{3,i}\cap G_i}$ and at most $\eps^3 m^2$ additional edges.
Thus with high probability $(U_{j,i},V_{j,i})_{H_{3,i}}$ is $(5\eps/2, \gamma
\beta /5)$-super-regular, i.e.~(ii) holds with high probability.

To prove (iii), observe that by ($a_2$) and Lemma~\ref{Properties after
definition of H's}(ii) $(U_{j,i},V_{j,i})_{H_3\cap G_i}$ (and thus
also $H_{3,i}\cap G_i$) has maximum degree at most $2\gamma \beta m$.
Moreover, every
edge in $E(H_3)\setminus E(G_i)$ has probability at most $1/r\le 5\beta/k$ of being assigned
to $H_{3,i}$. Since by Lemma~\ref{Properties after
definition of H's}(i) $H_3$ has maximum degree at most $2\gamma n$, this implies that
$H_{3,i}-E(G_i)$ has maximum degree at most $10\gamma\beta n/k$. Thus~(iii)
follows from Theorem~\ref{Chernoff} with room to spare.

In order to satisfy~(iv) we delete all the edges of $H_{3,i}$ which
do not `correspond' to an edge of~$R$.
\end{proof}
We choose $H_{3,1},\dots, H_{3,r}$ as in
Lemma~\ref{Distribution of edges of H_3 within the G_i's}. We now redefine
each $G_i$ by removing from it every
edge which belongs to one of the $H_j$'s. Observe that each $G_i$
still satisfies $(a_1)$ and $(a_4)$ and it also satisfies

\begin{itemize}
\item [($a_2'$)] For each $x \in V(G)$ the
degree of $x$ in $G_i$ is either~0 if $x \in V_{0i}$ or $\beta(1 \pm
4\gamma) m$ otherwise;
\item [($a_3'$)] For each $1 \leq j \leq k'$
the pair $(U_{j,i},V_{j,i})$ is $(2\eps, \beta(1 -
4\gamma))$-super-regular,
\end{itemize}
instead of $(a_2)$ and $(a_3)$ respectively. Indeed, $(a_2')$
follows from $(a_2)$ and Lemma~\ref{Properties after definition of
H's}(ii) while $(a_3')$ follows from $(a_3)$ and
Lemma~\ref{Properties after definition of H's}(iv),(vi). Moreover,
since we have removed the edges of $H_1,H_2$ and $H_3$ from the
$G_i$'s we have
\begin{itemize}
\item [($a_5$)] $G_1, \ldots , G_r,H_1,H_2,H_3$ are edge-disjoint.
\end{itemize}

\subsection{Adding edges between $V_{0i}$ and $G_i \setminus V_{0i}$}

Our aim in this subsection is to add edges from $G \setminus (G_1
\cup \cdots \cup G_r \cup H_2 \cup H_3)$ into the $G_i$'s so that
for each $1 \leq i \leq r$ we have the following new properties:

\begin{itemize}
\item[($a_{2.1}$)] For each $x \in V(G)$, we have $d_{G_i}(x)
= (1 \pm 5\gamma)\beta m$;
\item[($a_{2.2}$)] For each $x \in G_i \setminus
V_{0i}$, we have $d_{V_{0i}}(x) \leq \sqrt{\eps}\beta m$,
\end{itemize}
instead of $(a_2')$. We will also guarantee that no edge will be
added to more than one of the $G_i$'s. In particular, instead of
$(a_5)$ we will now have

\begin{itemize}
\item [($a_5'$)] $G_1, \ldots , G_r,H_2,H_3$ are edge-disjoint.
\end{itemize}
Moreover, all edges added to $G_i$ will have one endpoint in
$V_{0i}$ and the other endpoint in $G_i \setminus V_{0i}$. In
particular $(a_1)$ and $(a_3')$ will still be satisfied while
instead of $(a_4)$ we will have

\begin{itemize}
\item [($a_4'$)] Every edge of $G_i$ lies either in a pair of the form
$(V_{0i},U)$ where $U$ is a cluster of $G_i$ (i.e. $U=U_{i,j}$ or $U=V_{i,j}$ for some $1\leq j\leq k¢$) or in a pair of the form
$(U_{j,i},V_{j,i})$ for some $1 \leq j \leq k'$.
\end{itemize}

We add the edges as follows: Firstly, for each vertex $x$ of $G$, we let
$L_x = \{i : x \in V_{0i} \}$. The distribution of the new edges
incident to $x$ will depend on the size of $L_x$. Let us write
$\ell_x = |L_x|$ and let $A = \{x : \ell_x \leq \gamma n/4 \beta
m\}$ and $B = V(G) \setminus A = \{x : \ell_x > \gamma n/4 \beta
m\}$.

We begin by considering the edges of $H_1$ incident to vertices of
$A$. For every such edge $xy$, we choose one of its endpoints
uniformly and independently at random. If the chosen endpoint, say
$x$, does not belong to $A$, then we do nothing. If it does belong
to $A$ then we will assign $xy$ to at most one of the $G_i$'s
for which $i \in L_x$. For each $i \in L_x$, we assign $xy$
to $G_i$ with probability $2\beta m/d_{H_1}(x)$. So the
probability that $xy$ is not assigned to any~$G_i$ is $1 -
\frac{2 \ell_x \beta m}{d_{H_1}(x)}$. (Moreover, this assignment is
independent of any previous random choices.) Observe that since
$\delta(G) \geq (1/2 + \alpha)n$,
Lemma~\ref{Properties after definition of H's}(i) implies that $\frac{2
\ell_x \beta m}{d_{H_1}(x)} \leq \frac{\gamma n}{2 d_{H_1}(x)} \leq
1$, so this distribution is well defined. Finally, we remove all
edges that lie within some $V_{0i}$, so that each $G_i[V_{0i}]$
becomes empty.

\begin{lemma}\label{Properties after adding edges from H_1 to the G_i's}
With  probability at least $2/3$ the following properties hold:
\begin{itemize}
\item[(i)] For every $i$ and every $x \in V_{0i} \cap A$
we have $|d_{G_i}(x) - \beta m| \leq 8 \eps \beta m$;
\item[(ii)] For every $i$ and every $x \in G_i \setminus V_{0i}$ we
have $|N_{G_i}(x) \cap (V_{0i} \cap A)| \leq 9 \eps \beta m$.
\end{itemize}
\end{lemma}

\begin{proof}
The results will follow by applications of Theorem~\ref{Chernoff}.
\begin{itemize}
\item[(i)] For every $x \in V_{0i} \cap A$ and every edge $xy$ of
$H_1$ with $y \notin V_{0i}$, the probability that $xy$ is assigned to
$G_i$ is exactly $\beta m/d_{H_1}(x)$. Indeed, with probability
$1/2$, the endpoint $x$ of $xy$ is chosen and then independently
with probability $2\beta m/d_{H_1}(x)$ we assign $xy$ to $G_i$. Observe
that since $y \notin V_{0i}$, if the endpoint $y$ of $xy$ was
chosen, then $xy$ cannot be assigned to $G_i$. Thus, the expected
size of $d_{G_i}(x)$ is $\beta m \frac{d_{H_1 \setminus
V_{0i}}(x)}{d_{H_1}(x)}$, which by
Lemma~\ref{Properties after definition of H's}(i),(iii) is at most $\beta m$
and at least
$$
\beta m \left(1 - \frac{\gamma d_{V_{0i}}(x) + \zeta
n}{\gamma d_{G}(x) - \zeta n} \right) \stackrel{(\ref{v0i})}{\geq} (1 - 7\eps)\beta m.
$$
Thus by Theorem~\ref{Chernoff}, the probability that the required property
fails is at most $2rn \exp{\left(-\frac{\eps^2\beta^2 m^2}{3 \beta m}
\right)} \le 1/6$.

\item[(ii)] By Lemma~\ref{Properties after definition of H's}(iii) and~\eqref{v0i},
we have that $|N_{H_1}(x) \cap (V_{0i} \cap A)| \leq \gamma|V_{0i}|
+ \zeta n \leq 4\gamma \eps n$. By
Lemma~\ref{Properties after definition of H's}(i), every edge $xy$ of
$H_1$ with $y \in V_{0i} \cap A$ has probability at most $\beta
m/d_{H_1}(y) \leq 2\beta m/\gamma n$ of appearing in $G_i$. Since
all such events are independent, by Theorem~\ref{Chernoff} the
probability that (ii) fails is at most $2rn \exp{\left(-\frac{
\eps^2\beta^2 m^2}{24 \eps \beta m} \right)} \le 1/6$. \qedhere
\end{itemize}
\end{proof}
We now consider the edges of $H_1$ incident to vertices of $B$.
Observe that on the one hand we have $\sum |V_{0i}| \geq |B|
\frac{\gamma n}{4 \beta m}$. On the other hand, (\ref{lowerr}) and~(\ref{v0i}) imply that
$\sum |V_{0i}| \leq  \frac{3\eps n k}{\beta}$. Thus $|B| \leq
12\eps n/\gamma$.

For each $x \in B$, let $E(x)$ be the set of all edges of the form
$xy$ of $G$ such that $xy$ does not belong to any of the $G_i$'s or
any of the $H_j$'s and moreover $y \notin B \cup V_0$. By
definition we have that all the $E(x)$ are disjoint. Moreover,
using $(a'_2)$ and Lemma~\ref{Properties after definition of H's}(i)
\[
|E(x)| \geq \delta - (r - \ell_x)(1+4\gamma)\beta m-3(\gamma + \zeta)n -
\frac{12\eps}{\gamma}n - \eps n \stackrel{(\ref{lowerr})}{\geq} \ell_x \beta m.
\]
For each $x \in B$, we pick a subset $E'(x)$ of $E(x)$ of size exactly
$\ell_x \beta m$. We now assign each edge in $E'(x)$ uniformly at
random to the $\ell_x$ $G_i$'s with $i \in L_x$. Again, we then remove
from $G_i$ any edge that lies within $V_{0i}$, so that $G_i[V_{0i}]$
is still empty.

\begin{lemma}\label{Properties after adding edges from all E(x) into
the G_i's} With probability at least $2/3$ the following properties hold:
\begin{itemize}
\item[(i)] For every $i$ and every $x \in V_{0i} \cap B$
we have $|d_{G_i}(x) - \beta m| \leq \sqrt{\eps}\beta m$;
\item[(ii)] For every $i$ and every $x \in G_i \setminus V_{0i}$ we
have $|N_{G_i}(x) \cap (V_{0i} \cap B)| \leq \sqrt{\eps} \beta m/2$.
\end{itemize}
\end{lemma}

\begin{proof}
The results will follow by applications of Theorem~\ref{Chernoff}.
\begin{itemize}
\item[(i)] For every $x \in B$ and every $y \notin V_{0i}$ with $xy\in E'(x)$,
the probability that $xy$ is assigned to $G_i$ is
exactly $1/\ell_x$. Since also $|E'(x)|-| V_{0i}| \geq \ell_x
\beta m - 3\eps n$ by \eqref{v0i}, the expected size of $d_{G_i}(x)$ is at most
$\beta m$ and at least%
    \COMMENT{$\beta m-3\eps n/\ell_x\ge \beta m-3\eps n4\beta m/\gamma n
\ge (1 - \frac{\sqrt{\eps}}{2})\beta m$}
$(1 - \sqrt{\eps}/2)\beta m$.
So by Theorem~\ref{Chernoff}, the probability of
failure is at most $2rn \exp{\left(- \frac{\eps \beta^2
m^2}{12\beta m} \right)} \le 1/6$.
\item[(ii)] We have that $|V_{0i} \cap B| \leq |B| \leq
12\eps n/\gamma$ and every edge $yx$ with $y \in V_{0i} \cap B$
has probability either $1/\ell_y$ or 0 of appearing in $G_i$
independently of the others.
So
$$
\E \left( |N_{G_i}(x) \cap (V_{0i} \cap B)| \right)
\le \frac{|B|}{\ell_y} \le \frac{12\eps n}{\gamma} \cdot \frac{4 \beta m}{\gamma n}
\le  \frac{\sqrt{\eps}}{4} \beta m.
$$
So by Theorem~\ref{Chernoff}, the
probability that (ii) fails is at most $2rn \exp{\left(- \frac{
\sqrt{\eps} \beta m}{12} \right)} \le 1/6$. \qedhere
\end{itemize}
\end{proof}

Thus we can make a choice of edges which we add to the $G_i$ so that both properties
in Lemmas~\ref{Properties after adding edges from H_1 to the G_i's}
and~\ref{Properties after adding edges from all E(x) into
the G_i's} hold. This in turn implies that the properties $(a_{2.1})$, $(a_{2.2})$ as
well as the other properties stated at the beginning of the subsection
are satisfied.

\subsection{Adding edges between the clusters of $G_i$}

Recall that by $(a_{2.1})$ every vertex of $G_i$ has degree $(1 \pm
5\gamma)\beta m$. We would like to almost decompose each $G_i$ into
Hamilton cycles. This would definitely be sufficient to complete the
proof of Theorems~\ref{Almost regular} and~\ref{Minimum
Degree}(ii). The first step would be to extract from $G_i$ an
$s$-regular spanning subgraph $S_i$ where $s$ is close to $(1 \pm
5\gamma)\beta m$. Observe that if $G_i$ does not have such an $S_i$,
then definitely it cannot be almost decomposed into Hamilton cycles.
It turns out that at the moment, we cannot guarantee the existence
of such an $S_i$. For example, consider the case when there are no
edges between the vertices of $V_{0i}$ and the vertices in clusters of
the form $U_{j,i}$ (i.e.~all vertices incident to $V_{0i}$ lie in the
$V_{j,i}$). This `unbalanced'  structure of $G_i$ implies that it cannot
 contain any regular spanning subgraph.

Our aim in this subsection is to use edges from $H_2$ in order to
transform the $G_i$'s so that they have some additional properties
which will guarantee the existence of~$S_i$. We will show that
adding only edges of $H_2$ to the $G_i$'s we can for each $1
\leq i \leq r$ guarantee the following new properties:

\begin{itemize}
\item[$(a_{2.1}')$] For each $x \in V(G)$, $d_{G_i}(x)
= (1 \pm 15\gamma)\beta m$;
\item[$(a_6)$] For all clusters $U\neq V$ of $G_i$ so that $U^R$ and
$V^R$ are adjacent in $R$ but not in~$M_i$, we have $|E_{G_i}(U,V)| \geq \beta \gamma
dm^2/8k$ and moreover for every $x \in U \cup V$ we also have
$|N_{G_i}(x) \cap (U \cup V)| \leq 10\beta \gamma m/k$.
\end{itemize}
No edge will be added to more than one of the $G_i$'s and so
(instead of $(a_5')$) we will have
\begin{itemize}
\item [($a_5''$)] $G_1, \ldots , G_r,H_3$ are edge-disjoint.
\end{itemize}
Finally, all edges added to $G_i$ will have both endpoints in
distinct clusters of $G_i$ and moreover for each $1 \leq j \leq k'$,
no edge will be added to $G_i$ between the clusters $U_{j,i}$ and
$V_{j,i}$. In particular, $(a_1),(a_{2.2})$ and $(a_3')$ will still
hold while instead of $(a_4')$ we will have
\begin{itemize}
\item [($a_4''$)] Every edge of $G_i$ lies in a pair of the form
$(V_{0i},U)$, where $U$ is a cluster of $G_i$, or a pair of the form
$(U,V)$, where $U$ and $V$ are clusters of $G_i$ with $U^R$ and $V^R$
adjacent in $R$.
\end{itemize}

For every pair of adjacent clusters $U$ and $V$ of $R$, we will
distribute the edges in $E_{H_2}(U,V)$ to the $G_i$ so that the
following lemma holds. It is then an immediate consequence that all
of the above properties are satisfied.

\begin{lemma}\label{Distribution of edges of H_2 within the G_i's}
Let $U$ and $V$ be adjacent clusters of $R$. Then we can assign some
of the edges of $H_2$ between $U$ and $V$ to the $G_i$ so that
every edge is assigned to at most one $G_i$ and moreover
\begin{itemize}
\item[(i)] If $UV$ is an edge of $M_i$, then no edge is assigned
to $G_i$. Otherwise, at least $\beta \gamma d m^2/8k$ edges are
assigned to $G_i$ and none of these edges has an endvertex in
$(U\setminus U(i))\cup (V\setminus V(i))$;
\item[(ii)] For every $x \in U(i) \cup V(i)$ at most $10 \beta \gamma m
/k$ edges incident to $x$ are assigned to $G_i$.
\end{itemize}
\end{lemma}

\begin{proof}
Given such $U,V$, we assign every edge of $E_{H_2}(U,V)$
independently and uniformly at random among the $G_i$'s. If an edge
assigned to $G_i$ is incident to
$(U\setminus U(i))\cup (V\setminus V(i))$ it is discarded. If
moreover $UV$ is an edge of $M_i$, then all edges assigned to $G_i$
are discarded.

Since $(U,V)_{G'}$ is $\eps$-regular of density at least $d$,
Lemma~\ref{Properties after definition of H's}(vii) implies that
$|E_{H_2}(U,V)| \geq \gamma d m^2/2$ and so by
Theorem~\ref{Chernoff}, the number of edges assigned to each $G_i$
is with high probability at least $\gamma d m^2/4r\ge \beta \gamma d m^2/4k$.
(The last inequality follows from~(\ref{lowerr}).) To prove~(i), it is
enough to show that (if $UV$ is not an edge of $M_i$ then) at most
half of these edges are discarded. Since $|U \setminus U(i)|,|V \setminus
V(i)| \leq 2\eps m$, there are at most $4\eps m^2$ such edges
which are incident in $G$ to a vertex of $(U\setminus U(i))\cup (V\setminus V(i))$. Of those, with
high probability at most $5\eps m^2/r\leq 25\eps \beta m^2/k$ are assigned to
$G_i$ and are thus discarded. To complete the proof, observe that by
Lemma~\ref{Properties after definition of H's}(v)
every vertex $x\in U$ has $|N_{H_2}(x) \cap V| \leq 3\gamma m/2$
(and similarly for every vertex $x\in V$), so by
Theorem~\ref{Chernoff} with high probability no vertex of $G_i$ is
incident to more than $2 \gamma m /r\le 10 \beta \gamma m /k$ assigned edges.
\end{proof}

\subsection{Finding the regular subgraph $S_i$}

Our aim in this subsection is to show that each $G_i$ contains a
regular spanning subgraph $S_i$ of even degree $s := \left(1 -
15\gamma \right)\beta m$. Moreover, for every cluster $V$
all its vertices have most of their neighbours in the cluster that $V$ is matched to in $M_i$
(see Lemma~\ref{Construction of S_i}).

To prove this lemma, we proceed as follows:
A result of Frieze and Krivelevich~\cite{Frieze&Krivelevich05} (based on the max-flow min-cut theorem)
implies that every pair $(U_{j,i},V_{j,i})$ contains a regular subgraph of degree close to $\beta m$.
However, the example in the previous subsection shows that
it is not possible to combine these to an $s$-regular spanning subgraph of $G_i$ due to the
the existence of the vertices in $V_{0i}$.
So in Lemma~\ref{Construction of T_i} we will first find a subgraph $T_i$ of $G_i$ where the vertices of $V_{0i}$
have degree $s$, every non-exceptional vertex has small degree in $T_i$
and moreover each pair $(U_{j,i},V_{j,i})$ will be balanced with respect to $T_i$
in the following sense:
the sum of the degrees of the vertices of $U_{i,j}$ in $T_i$ is equal to
the sum of the degrees of the vertices of $V_{i,j}$ in $T_i$.
We can then use the  following generalization (Lemma~\ref{Max-flow}, proved in~\cite{Kuhn&Osthus&Treglown})
of the result in~\cite{Frieze&Krivelevich05}: in each pair $(U_{j,i},V_{j,i})$
we can find a subgraph $\Gamma_{j,i}$ with prescribed degrees (as long as the prescribed degrees are not
much smaller than $\beta m$).
We then prescribe these degrees so that together with those in $T_i$ they
add up to $s$. So the union of the $\Gamma_{j,i}$ (over all $1\le j\le k'$)
and $T_i$ yields the desired $s$-regular subgraph $S_i$.
Note that since $S_i$ is regular, $(U_{j,i},V_{j,i})$ is balanced with respect to $S_i$ in the above sense
(i.e.~replacing $T_i$ with $S_i$).
Also, the pair will clearly be balanced with respect to $\Gamma_{j,i}$.
This explains why we needed to ensure that the pair is also balanced with respect to~$T_i$.

\begin{lemma}\label{Construction of S_i}
For every $1 \leq i \leq r$, $G_i$ contains a subgraph $S_i$ such
that
\begin{itemize}
\item[(i)] $S_i$ is $s$-regular, where $s := \left(1 -
15\gamma \right)\beta m$ is even;
\item[(ii)] For every $1 \leq j \leq k'$ and every $x \in U_{j,i}$ we have
$|N_{S_i}(x) \setminus V_{j,i}| \leq \eta \beta m$. Similarly,
$|N_{S_i}(x) \setminus U_{j,i}| \leq \eta \beta m$ for every $x \in V_{j,i}$.
\end{itemize}
\end{lemma}

As discussed above, to prove Lemma~\ref{Construction of S_i} we will show that
every $G_i$ contains a subgraph $T_i$ with the following properties:

\begin{lemma}\label{Construction of T_i}
Each $G_i$ contains a spanning subgraph $T_i$ such that
\begin{itemize}
\item [(i)] Every vertex $x$ of $V_{0i}$ has degree $s$;
\item [(ii)] Every vertex $y$
of $G_i \setminus V_{0i}$ has degree at most $\eta \beta m$;
\item [(iii)] For every $1 \leq j \leq
k'$, we have $\sum_{x \in U_{j,i}} d_{T_i}(x) = \sum_{x \in V_{j,i}}
d_{T_i}(x)$;
\item [(iv)] For every $1 \leq j \leq k'$, we have
$E_{T_i}(U_{j,i},V_{j,i}) = \emptyset$.
\end{itemize}
\end{lemma}

Having proved this lemma, we can use the following result
from~\cite{Kuhn&Osthus&Treglown} to deduce the existence of $S_i$.

\begin{lemma}\label{Max-flow}
Let $0 < 1/m' \ll \eps  \ll \beta' \ll \eta \ll \eta' \ll 1$.
Suppose that $\Gamma =(U,V)$ is an $(\eps , \beta')$-super-regular
pair where $|U|=|V|=m'$. Define $\tau := (1- \eta')\beta' m'$.
Suppose we have a non-negative integer $x_u \leq \eta \beta' m'$
associated with each $u \in U$ and a non-negative integer $y_v \leq
\eta\beta' m'$ associated with each $v \in V$ such that $\sum _{u
\in U} x_u = \sum _{v \in V} y_v$. Then $\Gamma$ contains a spanning
subgraph $\Gamma '$ in which $\tau -x_u$ is the degree of each $u
\in U$ and $\tau - y_v$ is the degree of each $v \in V$.
\end{lemma}

\begin{proof}[Proof of Lemma~\ref{Construction of S_i}]
To derive Lemma~\ref{Construction of S_i} from Lemmas~\ref{Construction of T_i}
and~\ref{Max-flow}, recall that by ($a'_3$) for each $1 \leq j \leq k'$ the pair $(U_{j,i},V_{j,i})$
is $(2\eps,(1-4\gamma)\beta)$-super-regular. Thus we can apply Lemma~\ref{Max-flow}
to $(U_{j,i},V_{j,i})$ with $2\eta$ playing the role of $\eta$ in the lemma,
$\beta':=(1-4\gamma)\beta$, $\eta':= 1 -
\frac{1-15 \gamma}{(1-4\gamma)(1-2\eps)}$, $m':=(1-2\eps)m$, $x_u = d_{T_i}(u)$ for
every $u \in U_{j,i}$ and $y_v = d_{T_i}(v)$ for every $v \in
V_{j,i}$. Observe that with this value of $\eta'$, we have $\tau
= \left(1 - 15 \gamma \right)\beta m=s$.
Lemma~\ref{Construction of T_i}(ii) implies that for each $u \in
U_{j,i}$ and each $v \in V_{j,i}$ we have $2\eta \beta' m' =
2\eta (1-4\gamma)(1-2\eps)\beta m \geq  \eta \beta m \geq
x_u,y_v$. Lemma~\ref{Construction of T_i}(iii) implies
that $\sum _{u \in U} x_u = \sum _{v \in V} y_v$. Thus the
conditions of Lemma~\ref{Max-flow} hold and we obtain a subgraph
$\Gamma_{j,i}$ of $(U_{j,i},V_{j,i})$ in which every $u \in U_{j,i}$
has degree $s - x_u$ and every $v \in V_{j,i}$ has degree $s - y_v$.
It follows from Lemma~\ref{Construction
of T_i}(i),(ii) and (iv) that $S_i = T_i \cup \left(\bigcup_{j=1}^{k'} \Gamma_{j,i}
\right)$ is as required in Lemma~\ref{Construction of S_i}.
\end{proof}

\begin{proof}[Proof of Lemma~\ref{Construction of T_i}]
We give an algorithmic construction of $T_i$. We begin by
arbitrarily choosing $s$ edges (of $G_i$) incident to each vertex
$x$ of $V_{0i}$. Recall that by $(a_{2.2})$ this means that every vertex of $G_i
\setminus V_{0i}$ currently has degree at most $\sqrt{\eps} \beta m$.
Let us write $u_{j,i} := \sum_{x \in U_{j,i}} d_{T_i}(x)$ and
$v_{j,i} := \sum_{x \in V_{j,i}} d_{T_i}(x)$. Note that these values
will keep changing as we add more edges from $G_i$ into $T_i$
and we currently have $|u_{j,i}-v_{j,i}|\leq \sqrt{\eps}\beta m^2$.

\medskip

\noindent \textbf{Step 1.} \emph{By adding at most $k'$ more edges,
we may assume that for every $1 \leq j \leq k'$,  $u_{j,i} -
v_{j,i}$ is even.}

\smallskip

\noindent

To prove that this is possible, take any $j$ for which $u_{j,i} -
v_{j,i}$ is odd and observe that there is a $j' \neq j$ for which
$u_{j',i} - v_{j',i}$ is also odd. This holds because $s$ is even and so there is an even
number of edges between $V_{0i}$ and $G_i \setminus V_{0i}$. Let $V$
be a cluster of $R$ which is a common neighbour (in $R$) of
$U_{j,i}^R$ and $U_{j',i}^R$ and which is distinct from $V_{j,i}^R$
and $V_{j',i}^R$. The existence of $V$ is guaranteed by the degree
conditions of $R$ (see Lemma~\ref{Degrees of R}(i)). Now we take an edge of $G_i$ between $V(i)$ and
$U_{j,i}$ not already added to $T_i$ and add it to $T_i$. We also
take an edge of $G_i$ between $V(i)$ and $U_{j',i}$ not already
added to $T_i$ and add it to $T_i$. This makes the differences for
$j$ and $j'$ even and preserves the parity of all other differences. So we can
perform Step 1.

\smallskip

In each subsequent step, we will take two clusters $U$ and $V$ of
$G_i$ and add several edges between them to $T_i$, these edges are
chosen from the edges of $G_i$ which are not already used. The
clusters $U^R$ and $V^R$ will be adjacent in $R$ but not in $M_i$, so condition (iv) will remain true.
We will only add at most $\beta \gamma d m^2/20k$ edges at each step
and we will never add edges between $U$ and $V$ more than twice.
Condition~$(a_6)$ guarantees that we have enough edges for this. (Recall that we have
already added at most $k'$ edges between each pair of clusters.) At the end of all
these steps condition (iii) will hold. Moreover, we will guarantee
that no cluster $U$ is used in more than $2\eta k$ of these
steps and so by~$(a_6)$ the degree of each vertex in $T_i$ will not be increased
by more than $20\beta\eta \gamma m\ll \eta \beta m$ and so condition (ii) will also be
satisfied.

We call a cluster $U$ of $G_i$ \emph{bad} if it is already used in
more than $\eta k$ of the above steps. We will also guarantee
that the number of the above steps is at most $\eta^2 k^2/2$. Since in each step
we use two clusters, this will imply that at each step
there are at most $\eta k$ bad clusters.

Let us now show how all the above can be achieved. Let us take a $j$
for which $u_{j,i} \neq v_{j,i}$, say $u_{j,i} < v_{j,i}$. (The case
$u_{j,i} > v_{j,i}$ is identical and will thus be omitted.) Since
by Lemma~\ref{Degrees of R}
the minimum degree of $R$ is at least $(\delta/n - 2d)k/\beta$ and
since there are no more than $1/\beta$ parallel edges between any
two vertices of $R$, it follows that there are at least $(\delta/n -
1/2 - 2d)k \geq \alpha k/2$ indices $j'$ such that $U_{j,i}^R$ is
adjacent to both $U_{j',i}^R$ and $V_{j',i}^R$ in $R$. Since there
are at most $\eta k$ bad clusters, there are at least $\alpha
k/3$ indices $j'$ such that $U_{j,i}^R$ is adjacent to both
$U_{j',i}^R$ and $V_{j',i}^R$ in $R$ and moreover none of
$U_{j',i}^R$ and $V_{j',i}^R$  is bad. As long as $v_{j,i} - u_{j,i}
> \beta \gamma d m^2/10k$, we add exactly $\beta \gamma d m^2/20k$
edges between $U_{j,i}$ and $U_{j',i}$ and exactly $\beta \gamma d
m^2/20k$ edges between $U_{j,i}$ and $V_{j',i}$. Note that this
decreases the difference $v_{j,i} - u_{j,i}$ and leaves all other
differences the same. Finally, if $0 < v_{j,i} - u_{j,i} < \beta
\gamma d m^2/10k$ then we carry out the same step except that we add
$(v_{j,i} - u_{j,i})/2$ edges between $U_{j,i}$ and $U_{j',i}$ and
between $U_{j,i}$ and $V_{j',i}$ instead. (Recall that Step 1
guarantees that $v_{j,i} - u_{j,i}$ is even.) As observed at the beginning of the proof,
the initial difference between $u_{j,i}$ and $v_{j,i}$ is at most
$\sqrt{\eps} \beta m^2$. This might have increased to at most
$2\sqrt{\eps}  \beta m^2$ after performing Step~1. Thus it takes at most
$20\sqrt{\eps}k/\gamma d+1\ll \eta k$ steps to make
$u_{j,i}$ and $v_{j,i}$ equal and so we may choose a different index $j'$
in each of these steps.

We repeat this process for all $1 \leq j \leq k'$. Obviously, (iii)
holds after we have considered all such $j$'s. It remains to check
that all the conditions that we claimed to be true throughout the
process are indeed true. As for each $j$ it takes at most $\eta k$ steps to make
$u_{j,i}$ and $v_{j,i}$ equal, the total number of steps
is at most $\eta^2 k^2/2$. Since moreover, a cluster $U_{j,i}$ or
$V_{j,i}$ is used in a step only when $j$ is considered or when it
is not bad, it is never used in more than $2\eta k$ steps, as
promised.
\end{proof}

\subsection{Choosing an almost 2-factor decomposition of $S_i$}

Since each $S_i$ is regular of even degree, by
Theorem~\ref{Petersen} we can decompose it into 2-factors.
Our aim will be to use the edges of $H_{3,i}$ to
transform each 2-factor in this decomposition into a Hamilton cycle.
To achieve this, we need each 2-factor in the decomposition
to possess some additional properties.
Firstly, we would like each 2-factor to contain $o(n)$ cycles.
To motivate the second property, note that
by Lemma~\ref{Construction of S_i}(ii),
most edges of $S_i$ go between pairs of clusters $(U_{j,i},V_{j,i})$.
So one would expect that this is also the case for a typical $2$-factor $F$.
We will need the following stronger version of this property: for every pair
$(U_{j,i},V_{j,i})$ of clusters of~$G_i$ and every vertex $u\in U_{j,i}$,
most of its $S_i$-neighbours in $V_{j,i}$ have both their $F$-neighbours in
$U_{j,i}$ (and similarly for every $v \in V_{j,i}$).
We will also need the analogous property with $S_i$ replaced by $H_{3,i}$.

The following lemma tells us that we can achieve the above
properties if we only demand an almost 2-factor decomposition.

\begin{lemma}\label{Properties of 2-factor decomposition}
$S_i$ contains at least $\left(1 - \sqrt{\gamma} \right) \frac{\beta
m}{2}$ edge-disjoint $2$-factors such that for every such $2$-factor $F$ the
following hold:
\begin{itemize}
\item[(i)] $F$ contains at most $n/(\log n)^{1/5}$ cycles;
\item[(ii)] For every $1 \leq j \leq k'$ and every $u\in U_{j,i}$,
the number of $H_{3,i}$-neighbours of $u$ in $V_{j,i}$ which have an $F$-neighbour outside $U_{j,i}$ is at
most   $\gamma^3 \beta m$ (and similarly for the $H_{3,i}$-neighbours in $U_{j,i}$ of each $v \in V_{j,i}$).
\item[(iii)] For every $1 \leq j \leq k'$ and every $u\in U_{j,i}$,
the number of $S_i$-neighbours of $u$ in $V_{j,i}$ which have an $F$-neighbour outside $U_{j,i}$ is at
most   $\gamma^3 \beta m$ (and similarly for the $S_i$-neighbours in $U_{j,i}$ of each $v \in V_{j,i}$).
\end{itemize}
\end{lemma}

The proof of Lemma~\ref{Properties of 2-factor decomposition} will rely on the following lemma
from~\cite{Kuhn&Osthus&Treglown}. This lemma is in turn based on
a result in~\cite{randmatch} whose proof relies on a
probabilistic approach already used in~\cite{Frieze&Krivelevich05}.
A \emph{1-factor} in an oriented graph~$D$ is a collection of disjoint
directed cycles covering all the vertices of~$D$.

\begin{lemma}\label{Existence 2-factor decomposition}
Let $0 < \theta_1,\theta_2,\theta_3 < 1/2$ be such that $\theta_1/\theta_3
\ll \theta_2$. Let $D$ be a $\theta_3 n$-regular oriented graph
whose order $n$ is sufficiently large. Suppose $A_1 , \ldots ,
A_{5n}$ are sets of vertices in $D$ with $|A_t| \geq n^{1/2}$. Let
$H$ be an oriented subgraph of $D$ such that $d^+_H(x),d^-_H(x) \leq
\theta_1 n$ for all $x \in A_t$ and each $t$. Then $D$ has a
$1$-factor $F$ such that
\begin{itemize}
\item[(i)] $F$ contains at most $n/(\log n)^{1/5}$ cycles;
\item[(ii)] For each $t$, at most $\theta_2 |A_t|$ edges of $H \cap
F$ are incident to $A_t$.
\end{itemize}
\end{lemma}

\begin{proof}[Proof of Lemma~\ref{Properties of 2-factor
decomposition}]

We begin by choosing an arbitrary orientation $D$ of $S_i$ with
the property that every vertex has indegree and outdegree equal to
$s/2$. The existence of such an orientation follows e.g.~from
Theorem~\ref{Petersen}.\COMMENT{Or quote Euler.} We repeatedly
extract 1-factors of $D$ satisfying the properties of
Lemma~\ref{Properties of 2-factor decomposition} as follows: Suppose
we have extracted some 1-factors from $D$ and we are left with a
$\theta_3 n$-regular oriented graph $D$, where $\theta_3 \geq
\sqrt{\gamma} \beta m/4n$.

For the sets $A_t$, we take all sets of the
form $N_{H_{3,i}}(u)\cap V_{j,i}$ and all sets of the
form $N_{S_i}(u)\cap V_{j,i}$ (for all $u\in U_{j,i}$ and $j=1,\dots,k'$)
as well as all sets of the form $N_{H_{3,i}}(v)\cap U_{j,i}$ and
all sets of the form $N_{S_i}(v)\cap U_{j,i}$ (for all $v\in V_{j,i}$ and
$j=1,\dots,k'$).
Even though the number of these sets is less than $5n$, this is not a problem as for
example we might repeat each set several times. Lemmas~\ref{Distribution of edges of H_3 within the G_i's}(ii)
and~\ref{Construction of S_i}(ii) imply that these sets have size at least%
   \COMMENT{don't get $\gamma\beta m/5$ since the cluster size is not quite $m$}
$\gamma\beta m/6\gg n^{1/2}$.

For the subgraph $H$ of $D$ we take the graph consisting of all those edges
of $S_i$ which do not belong to some pair $(U_{j,i},V_{j,i})$. Then $d^+_H(x),d^-_H(x) \leq \theta_1
n$ for all $x\in A_t$ (and each~$t$), where by Lemma~\ref{Construction of S_i}(ii) we can
take $\theta_1 =\eta \beta m/n$.

Thus, taking $\theta_2 = \gamma^3$
all conditions of Lemma~\ref{Existence 2-factor
decomposition} are satisfied and so we obtain a 1-factor $F$ of $D$
satisfying all properties of Lemma~\ref{Properties of 2-factor decomposition}.
(The fact that $s \le \beta m$ and
Lemma~\ref{Distribution of edges of H_3 within the G_i's}(iii) imply that
the $A_t$ have size at most $\beta m$ and so~$F$ satisfies
Lemma~\ref{Properties of 2-factor decomposition}(ii) and~(iii).) It follows that we can keep extracting such
1-factors for as long as the degree of $D$ is at least
$\sqrt{\gamma} \beta m/4$ and in particular we can extract at least
$(1 - \sqrt{\gamma})\beta m/2$ such 1-factors as required.
\end{proof}

\subsection{Transforming the 2-factors into Hamilton cycles}

To finish the proof it remains to show how we can use (for each $i$)
the edges of $H_{3,i}$ to transform each of the 2-factors of $S_i$ created by
Lemma~\ref{Properties of 2-factor decomposition}
into a Hamilton cycle. By Lemma~\ref{Properties of 2-factor decomposition},
this will imply that the total number of edge-disjoint Hamilton cycles
we construct is $(1-\sqrt{\gamma})r \beta m/2$, which suffices to prove
Theorems~\ref{Almost regular} and~\ref{Minimum Degree}(ii).
To achieve the transformation of each $2$-factor into a Hamilton cycle,
we claim that it is enough to prove the
following theorem.
In conditions~(iv) and~(v) of the theorem we say that a pair of clusters $(A_i,A_j)$ of a graph $X$ is
\emph{weakly $(\eps,\eps')$-regular} in a subgraph $H$ of $X$ if
for every $U \subseteq A_i, V \subseteq A_j$ with $|U|,|V| \geq
\eps m$, there are at least $\eps'  m^2$ edges
between $U$ and $V$ in $H$.

Roughly speaking, we will apply the following theorem successively to the
$2$-factors $F$ in our almost-decomposition of $S_i$ and where  $H$ is the union of $H_{3,i}$
together with some additional edges incident to $V_{0i}$ .
However, this does not quite work -- between successive applications of the
theorem we will also need to add edges to $H$ which were removed from a previous $1$-factor $F$
when transforming $F$ into a Hamilton cycle.

\begin{theorem}\label{Transforming 2-factors into Hamilton cycles}
Let $1/n\ll 1/k\le \eps \ll \beta \ll \gamma\ll 1$. Let $m$
be an integer such that $(1-\eps)n \le mk\le n$.
Let $H$ be a graph on $n$ vertices and let $F$ be a $2$-factor so that $F$ and $H$ have the same
vertices but are edge-disjoint. Let $X:=F \cup H$.
Let $A_1,\ldots,A_k$ be disjoint subsets of $X$ of size $(1 -
2\eps)m$ and
let $B_1,\ldots,B_{k'},D_1,\ldots,D_{k'}$ be another
enumeration of the  $A_1,\ldots,A_k$. Suppose also that the following
hold:
\begin{itemize}
\item[(i)] $F$ contains at most $n/(\log n)^{1/5}$ cycles;
\item[(ii)] For each $1 \leqslant i \leqslant k'$ and for each vertex of
 $B_i$ the number of $H$-neighbours in $D_i$ having an $F$-neighbour outside $B_i$ is
at most $2\gamma^3 \beta m$ (and similarly for the vertices in $D_i$);
\item[(iii)] For every $1 \leqslant i \leqslant k'$, the pair
$(B_{i},D_{i})_H$ is $(3\eps,\gamma \beta /6)$-super-regular;
\item[(iv)] For every $1 \leqslant i \leqslant k$ and every $A_i$, there are at least
$(1 + \alpha) k'$ distinct $j$'s with $1 \leq j \leq k$ such that
$(A_i,A_j)$ is weakly $(\eps,\eps^3/k)$-regular in $H$;
\item[(v)] For every $1 \leq i < j \leq k$, if there is
an edge in $X$ between $A_i$ and $A_j$ then $(A_i,A_j)$ is weakly $(\eps,\eps^3/k)$-regular in
$H$;
\item[(vi)] For every vertex $x \in V(X) \setminus (A_1 \cup \ldots \cup
A_k)$, both $F$-neighbours of $x$ belong to $A_1 \cup \ldots \cup A_k$.
\item[(vii)] Every vertex $x \in V(X) \setminus (A_1 \cup \ldots \cup
A_k)$ has degree at least $\alpha n/6$ in $H$ and every $H$-neighbour of $x$ lies
in $A_1 \cup \ldots \cup A_k$.
\end{itemize}
Then there is a Hamilton cycle $C$ of $X$ such that  $|E(C) \triangle E(F)| \leq 25 n/(\log n)^{1/5}$.
\end{theorem}
To see that it is enough to prove the above theorem, suppose we have
already transformed all 2-factors of $S_1,\ldots,S_{i-1}$ guaranteed
by Lemma~\ref{Properties of 2-factor decomposition} into edge-disjoint Hamilton
cycles such that for each $1\leq j\leq i-1$ the Hamilton cycles
corresponding to the $2$-factors of $S_j$ lie in $G \setminus
\bigcup_{j' > j} \left(G_{j'}\cup H_{3,j'} \right)$. Moreover,
suppose that we have also transformed $\ell$ of the $2$-factors of
$S_i$, say $F_1,\ldots, F_{\ell}$, into edge-disjoint Hamilton cycles
$C_1,\ldots,C_{\ell}$ such that $C_j\subseteq G \setminus
\bigcup_{i' > i} \left(G_{i'}\cup H_{3,i'} \right)$ and $|E(C_j)
\triangle E(F_j)| \leq 25 n/(\log n)^{1/5}$ for all $1\leq j\leq
\ell$. Obtain $H^*_1$ from $H_{3,i}$ as follows:
\begin{itemize}
\item[($b_0$)] add all those edges of $G$ between $V_{0i}$ and $V(G)\setminus V_{0i}$
which do not belong to any
$G_j \cup H_{3,j}$ with $j \geq i$ or to any Hamilton cycle already created.
\end{itemize}
Suppose that we have inductively defined graphs $H^*_1,\dots,H^*_{\ell}$
such that $C_j\subseteq H^*_j\cup F_j$ for all $1\le j\le \ell$. Define $H^*_{\ell+1}$
as follows:
\begin{itemize}
\item[($b_1$)] remove all edges in $E(C_\ell)\setminus E(F_\ell)$ from $H^*_\ell$;
\item[($b_2$)] add all edges in $E(F_\ell)\setminus E(C_\ell)$ to $H^*_\ell$.
\end{itemize}
Let $F_{\ell+1}$ be one of the 2-factors of $S_i$ as constructed in
Lemma~\ref{Properties of 2-factor decomposition} which is distinct
from $F_1, \ldots, F_{\ell}$. Finally, let $B_j = U_{j,i}$ and $D_j
= V_{j,i}$ for $1 \leq j \leq k'$. We claim that all conditions of
Theorem~\ref{Transforming 2-factors into Hamilton cycles} hold (with
$H^*_{\ell+1}$ and $F_{\ell+1}$ playing the roles of $H$ and $F$).
Indeed, property (i) follows from Lemma~\ref{Properties of 2-factor
decomposition}(i). Since $N_{H^*_{\ell+1}}(u)\cap V_{j,i}\subseteq
(N_{H_{3,i}}(u) \cup N_{S_i}(u))\cap V_{j,i}$ for every $u\in
U_{j,i}$ (note that this is not necessarily true for $u \in
V_{0,i}$), property~(ii) follows from Lemma~\ref{Properties of
2-factor decomposition}(ii) and~(iii). To see that property (iii)
holds, recall that by Lemma~\ref{Distribution of edges of H_3 within
the G_i's}(ii) we have that for every $1 \leqslant j' \leqslant k'$
the pair $(B_{j'},D_{j'})_{H_{3,i}}$ is $(5\eps/2,\gamma \beta
/5)$-super-regular. Since also $|E(C_j) \triangle E(F_j)| \leq 25
n/(\log n)^{1/5}$ for each $1 \leq j \leq \ell$, we have
$|E(H^*_{\ell+1}\setminus V_{0i}) \triangle E(H_{3,i}\setminus
V_{0i})| \leq 25 n^2/(\log n)^{1/5}$ and so
$(B_{j'},D_{j'})_{H^*_{\ell+1}}$ is $3\eps$-regular of density at
least $\gamma \beta /6$. To prove that the pair is  even
$(3\eps,\gamma \beta/6)$-super-regular, it suffices to show that for
any $x \in B_{j'}$ we have
\begin{equation} \label{superreg}
|N_{H^*_{\ell+1}}(x) \cap D_{j'}| \ge \gamma \beta m/6.
\end{equation}
(A bound for the case $x \in D_{j'}$ will follow in the same way.)
To prove~(\ref{superreg}), suppose that the degree of $x$ in $(B_{j'},D_{j'})_{H^*_{\ell+1}}$
was decreased by one compared to $(B_{j'},D_{j'})_{H^*_\ell}$ due to~($b_1$).
This means that an edge $xy$ of $(B_{j'},D_{j'})_{H^*_\ell}$ was inserted into $C_\ell$.
But since $F_\ell$ and $C_\ell$ are both $2$-factors, this means that
an edge $xz$ from $F_\ell$ will be added to $H^*_{\ell}$ when forming $H^*_{\ell+1}$.
Note that $xz \in E(F_\ell) \subseteq E(S_i)$ and
by our assumption on the degree of $x$,  we have $z \notin D_{j'}$.
If the degree decreases by two of $x$, then the argument shows that
we will be adding two such edges $xz_1$ and $xz_2$ to $H^*_\ell$ when forming $H^*_{\ell+1}$.
But since Lemma~\ref{Construction of S_i}(ii) implies that
$|N_{S_i}(x) \setminus D_{j'}| \le \eta \beta m$, this can happen at most
$\eta \beta m$ times throughout the process of constructing $C_1,\dots,C_\ell$.
(Here we are also using the fact that the $F_j$ are edge-disjoint, so we will consider
such an edge $xz$ or $xz_i$ only once throughout.)
So
$$
|N_{H^*_{\ell+1}}(x) \cap D_{j'}| \ge  |N_{H_{3,i}}(x) \cap D_{j'}| -\eta \beta m
\ge \gamma \beta (1-2\eps)m/5 -\eta \beta m\ge \gamma \beta m/6,
$$
which proves~(\ref{superreg}) and thus (iii). Property (iv) follows from
Lemma~\ref{Distribution of edges of H_3 within the G_i's}(i) together
with the fact that $|E(H_{3,i}\setminus V_{0,i}) \triangle E(H^*_{\ell+1}\setminus V_{0i})| = o(n^2)$ and
the fact that the minimum degree of $R$ is at least
$(1+\alpha)k/2\beta$ (see Lemma~\ref{Degrees of R}).
Property (v) follows similarly since by~($a''_4$) each edge in
$E(F_{\ell+1})\subseteq E(G_i)$ between clusters corresponds to an edge of $R$ and since by
Lemma~\ref{Distribution of edges of H_3 within the G_i's}(iv) the analogue holds for
the edges of $H_{3,i}$. Property (vi) is an immediate
consequence of $(a_4'')$. To see that (vii) holds consider a vertex
$x \in V(X) \setminus (A_1 \cup \ldots \cup A_k)$.
By Lemma~\ref{Properties after definition of H's}(i) $x$ has degree at most $2\gamma n$ in $H_3$
and thus in the union of $H_{3,j}$ with $j \ge i$.
By $(a_{2.1}')$, $x$ has degree at most $(r-i+1)(1+15\gamma)\beta m$ in the union of the  $G_j$
with $j \ge i$.
The number of Hamilton cycles already constructed is at most $i(1-\sqrt{\gamma})\beta m/2$.
Furthermore, $x$ has at most $|V_{0i}|\le 3\eps n$ neighbours in $V_{0i}$.
So altogether the number of edges of $G$
incident to $x$ which are not included in $H^*_{\ell+1}$ due to ($b_0$) and ($b_1$)
is at most $2\gamma n+(r+1) (1+15\gamma) \beta m +3\eps n\le \delta -\alpha n/6$, where the inequality follows
from the bound on $\delta$ in (\ref{lowerr}). So the number of edges incident to $x$
in~$H^*_{\ell+1}$ is at least $\alpha n/6$. Moreover, by
Lemma~\ref{Distribution of edges of H_3 within the G_i's}(iv) and~($a''_4$) no neighbour
of $x$ in $H_{3,i}\cup G_i$ lies in $V_{0i}$ and thus the same is true
for every $H^*_{\ell+1}$-neighbour of $x$.

\subsection{Proof of Theorem~\ref{Transforming 2-factors into Hamilton
cycles}}\label{sec:finalpf}

In the proof of Theorem~\ref{Transforming 2-factors into Hamilton
cycles} it will be convenient to use the following special case of a
theorem of Ghouila-Houri~\cite{Ghouila-Houri60}, which is an analogue of Dirac's theorem for
directed graphs.%
\COMMENT{ makes the proof slightly longer but hopefully more transparent}
\begin{theorem}[\cite{Ghouila-Houri60}] \label{ghouila}
Let $G$ be a directed graph on $n$ vertices with minimum out-degree
and minimum in-degree at least $n/2$. Then $G$ contains a directed
Hamilton cycle.
\end{theorem}
We will also use the following
`rotation-extension' lemma which appears implicitly
in~\cite{Frieze&Krivelevich05} and explicitly (but for directed
graphs) in~\cite{Kuhn&Osthus&Treglown}. The directed version implies
the undirected version (and the latter is also simple to prove directly).
Given a path $P$ with endpoints in opposite clusters of an $\eps$-regular pair,
the lemma provides a cycle on the same vertex set by changing only a small number of
edges.

\begin{lemma}\label{Rotation-Extension}
Let $0 < 1/m \ll \eps \ll \gamma' < 1$ and let $G$ be a graph on $n
\geq 2m$ vertices. Let $U$ and $V$ be disjoint subsets of $V(G)$
with $|U| = |V| = m$ such that for every $S \subseteq U$ and every
$T \subseteq V$ with $|S|,|T| \geq \eps m$ we have $e(S,T) \geq
\gamma' |S||T|$. Let $P$ be a path in $G$ with endpoints $x$ and $y$
where $x \in U$ and $y \in V$. Let $U_P$ be the set of vertices of
$P$ which belong to $U$ and have all of their $P$-neighbours in $V$
and let $V_P$ be defined analogously. Suppose that $|N(x) \cap
V_P|,|N(y) \cap U_P| \geq \gamma' m$. Then there is a cycle $C$ in
$G$ containing precisely the vertices of $P$ and such that $C$
contains at most $5$ edges which do not belong to $P$.
\end{lemma}

\begin{proof}[Proof of Theorem~\ref{Transforming 2-factors into Hamilton
cycles}]

We will give an algorithmic construction of the Hamilton cycle.
Before and after each step of our algorithm we will have a spanning
subgraph $H'$ of $H$ and spanning subgraph $F'$ of $X$ which is a
union of disjoint cycles and at most one path such that $H'$ and
$F'$ are edge-disjoint. In each step we will add at most 5 edges
from $H'$ to $F'$ and remove some edges from $F'$ to obtain a new
spanning subgraph $F''$. The edges added to $F'$ will be removed
from $H'$ to obtain the new subgraph $H''$. It will turn out that
the number of steps needed to transform $F$ into a Hamilton cycle
will be at most $5 n/(\log n)^{1/5}$. This will complete the proof
of Theorem~\ref{Transforming 2-factors into Hamilton cycles}.

To simplify the notation we will always write $H$ and $F$ for these
subgraphs of $X$ at each step of the algorithm.
Also, let $g(n):=n/(\log n)^{1/5}$.
We call all the edges of the initial~$F$ \emph{original}.
At each step of the algorithm, we will write $B_i'$ for the set of
vertices $b\in B_i$ whose neighbours in the current graph $F$ both lie in $D_i$ and are joined to~$b$
by original edges (for each $1 \leq i \leq k'$).%
   \COMMENT{need this condition on the original edges to ensure that there is never
a vertex which belongs to $B'_i$ after some step but didn't belong to $B'_i$ before
this step. So dealing with $B'_{i+1}$ in Claim~1 (say) won't affect what we have done for
$B'_i$}
We define $D_i'$ similarly. So during the
algorithm the size of each $B'_i$ might decrease, but since we delete at most $25 g(n)$
edges from the initial $F$ during the algorithm, all but at most $50 g(n)$ vertices
of the initial $B'_i$ will still belong to this set at the end of the algorithm
(and similarly for each $D'_i$).

Since at each step of the algorithm the current $F$ differs from the initial one by at most
$25 g(n)$ edges (and so at most $25 g(n)$ edges have been removed from the initial~$H$),
we will be able to assume that at each step of the algorithm the following conditions
hold.
\begin{itemize}
\item[(a)] For each $1 \leqslant i \leqslant k'$ each vertex of
$B_i$ has at most $3\gamma^3 \beta  m$  $H$-neighbours in $D_i\setminus D'_i$
(and similarly for the vertices in $D_i$);
\item[(b)] For every $1 \leqslant i \leqslant k'$, the pair
$(B_{i},D_{i})_H$ is $(4\eps,\gamma \beta /7)$-super-regular;
\item[(c)] For every $1 \leqslant i \leqslant k$ and every $A_i$, there are at least
$(1 + \alpha) k'$ distinct $j$'s with $1 \leq j \leq k$ such that
$(A_i,A_j)$ is weakly $(\eps,\eps^3/2k)$-regular in $H$;
\item[(d)] For every $1 \leq i < j \leq k$, if there is
an edge in $X$ between $A_i$ and $A_j$ then
$(A_i,A_j)$ is weakly $(\eps,\eps^3/2k)$-regular in $H$;
\item[(e)] Every vertex $x \in V(X) \setminus (A_1 \cup \ldots \cup
A_k)$ has degree at least $\alpha n/7$ in $H$ and all $H$-neighbours
of $x$ lie in $A_1 \cup \ldots \cup A_k$.
\end{itemize}

Note that by (a) and (b) we always have
\begin{equation} \label{bidi}
|B_i'|,|D_i'| \geq (1-\gamma)m.
\end{equation}%
To see this, suppose that initially we have $|B_i \setminus B_i'| \ge \gamma m/2$.%
\COMMENT{need $/2$ here as cluster size is not quite $m$} Then by
(b) there is a vertex $x \in D_i$ which has at least $\gamma^2 \beta
m/20 > 3\gamma^3 \beta m$ $H$-neighbours in $B_i \setminus B_i'$,
contradicting (a). So~\eqref{bidi} follows since we have already
seen that all but at most $50g(n)$ vertices of the original set
$B'_i$ still belong to $B'_i$ at the end of
the algorithm.%
\COMMENT{this can also be proved in the same way as (a) but the current method seems simple enough.}

\medskip

\noindent \textbf{Claim 1.} \emph{After at most $g(n)$ steps, we may assume that
$F$ is still a $2$-factor and that for each $1 \leq i \leq k'$ there
is a cycle $C_i$ of $F$ which contains at least $\gamma \beta m/9$
vertices of $B_i'$ and at least $\gamma \beta m/9$ vertices of
$D_i'$.}

\smallskip

\noindent

Note that we may have $C_i = C_j$ even if $i \neq j$ (and similarly in the
later claims).
To prove the claim, suppose that $F$ does not contain such a cycle
$C_i$ for some given $i$. Let $C$ be a cycle of $F$ which contains
an edge $xy$ with $x \in B_i$ and $y \in D_i$.
Note that such a cycle exists by (\ref{bidi}).
Consider the path $P$ obtained from $C$ by
removing the edge $xy$. If $x$ has an $H$-neighbour $y'$ on another
cycle $C'$ of $F$ such that $y'$ has an $F$-neighbour $x'$
with $x' \in B_i$ then we replace the path $P$ and the cycle $C'$
with the path $x'C'y'xPy$. (Note that $x'$ will be one of the neighbours
of $y'$ on $C'$.) We view the construction of this path as carrying out one step of
the algorithm. Observe that we have only used one edge
from $H$ and we have reduced the number of cycles of $F$ by 1 when extending $P$. Let
us relabel so that the unique path of $F$ is called $P$ and its
endpoints $x$ and $y$ belong to $B_i$ and $D_i$ respectively.
Repeating this extension step for as long as possible, we may assume
that no $H$-neighbour of $x$ which is not on $P$ has an
$F$-neighbour in $B_i$ and similarly no $H$-neighbour of $y$ which
is not on $P$ has an $F$-neighbour in $D_i$. In particular, by (a) and (b),
$x$ has at least $\gamma \beta m/8$ $H$-neighbours in $V(P)\cap D'_i$, and similarly $y$
has at least $\gamma \beta m/8$ $H$-neighbours in $V(P)\cap B'_i$.
By Lemma~\ref{Rotation-Extension} (applied with $U:=B_i$, $V:=D_i$ and $G:=X$)
it follows that we can use at most 5 edges of $H$ to convert $P$
into a cycle $C_i$ (we view this as another step of the algorithm).
Note that $C_i$ satisfies the conditions of the claim.%
    \COMMENT{When applying Lemma~\ref{Rotation-Extension}, $|B'_i|$ and $|D'_i|$
might decrease by at most 10. So after this step we have a cycle containing at least
$\gamma \beta m/8-10$ vertices of $B'_i$ and at least $\gamma \beta m/8-10$ vertices of $D'_i$.
In each subsequent step (ie when we consider $i+1,i+2,\dots$) we might loose 10 further
vertices in both $B'_i$ and $D'_i$. But since we are doing at most $g(n)$ steps, at the end
we will still have a cycle containing at least $\gamma \beta m/9$ vertices from $B'_i$
and at least $\gamma \beta m/9$ vertices from $D'_i$. (However, we don't want to put all
this into the paper -- anyone who notices it will know how to fix it...)}
Since the number of cycles in $F$ is initially at most $g(n)$ and since a Hamilton cycle certainly
would satisfy the claim, the number of steps can be at most $g(n)$.

\medskip

\noindent \textbf{Claim 2.} \emph{After at most $g(n)$ further steps, we may assume
that $F$ is still a $2$-factor and that for each $1 \leq i \leq k'$
there is a cycle $C_i'$ of $F$ which contains all but at most
$4\eps m$ vertices of $B_i'$ and all but at most $4\eps m$ vertices
of $D_i'$.}

\smallskip

\noindent

Let $C_i$ be a cycle of $F$ which contains at least $\gamma \beta
m/9$ vertices of $B_i'$ and at least $\gamma \beta m/9$ vertices
of $D_i'$. Suppose there are at least $4 \eps m$ vertices of $B_i'$
not covered by $C_i$. Then (b) implies that there is a vertex $b \in
B_i'$, which is not covered by $C_i$ and a vertex $d \in D_i'$ which
is covered by $C_i$ such that $b$ and $d$ are neighbours in $H$. Let
$C'$ be the cycle containing $b$ and let $x$ be any neighbour of $b$
on $C'$ and $y$ any neighbour of $d$ on $C_i$. Then removing the
edges $bx$ and $dy$ and adding the edge $bd$ we obtain the path
$xC'bdC_iy$ (see Figure~\ref{Rotation1}).
\begin{figure}[h]
\includegraphics[scale=1]{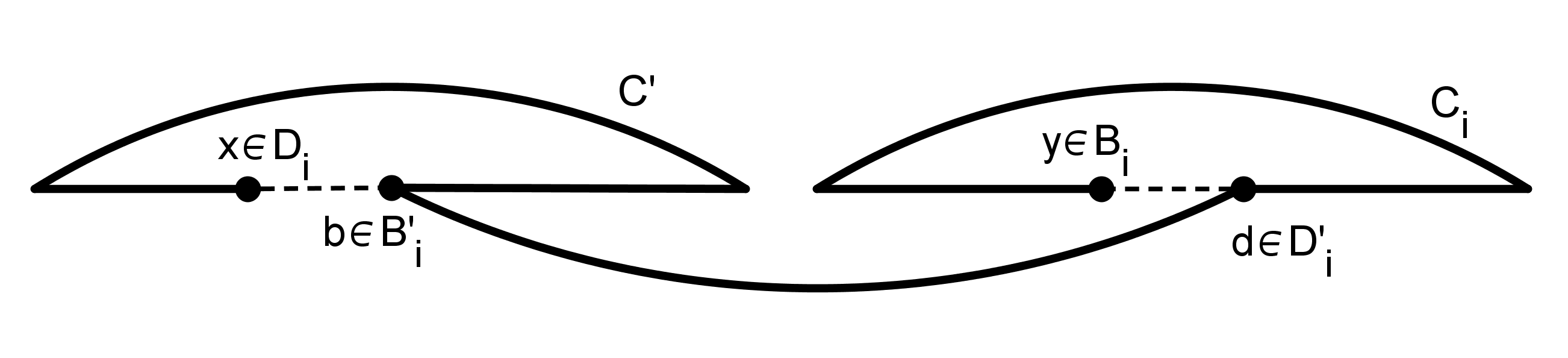}
\caption{Extending $C_i$ to include more vertices from $B_i' \cup
D_i'$.} \label{Rotation1}
\end{figure}

Since $x \in D_i$ and $y \in B_i$ (as $b\in B'_i$ and $d\in D'_i$)
we can repeat the argument in the
previous claim to extend this path into a larger path if necessary
and then close it into a cycle. As long there are at least $4\eps m$
vertices of $B_i'$ not covered by the cycle or at least $4\eps m$
vertices of $C_i'$ not covered by the cycle we can repeat the above
procedure to extend this into a larger cycle. Thus we can obtain a
cycle $C_i'$ with the required properties.
The bound on the number of steps follows as in Claim~1.

\medskip

\noindent \textbf{Claim 3.} \emph{After at most $g(n)$ further steps, we may assume
that $F$ is still a $2$-factor and that for each $1 \leq i \leq k'$
there is a cycle $C_i''$ of $F$ which contains all vertices of $B_i'
\cup D_i'$.}

\smallskip

\noindent

Let $C_i'$ be the cycle obtained in the previous claim and suppose
there is a vertex $b \in B'_i$ not covered by $C_i'$. By (a) and (b) it
follows that $b$ has at least $\gamma \beta m/8$ $H$-neighbours in $V(C_i')\cap D'_i$.
Let $d$ be such an $H$-neighbour of~$b$. Repeating the
procedure in the proof of the previous claim, we can enlarge $C'_i$ into a
cycle containing $b$. Similarly we can extend the cycle to include
any $d \in D_i'$, thus proving the claim.

\medskip

\noindent \textbf{Claim 4.} \emph{After at most $g(n)$ further steps, we may assume
that $F$ is still a $2$-factor and that for each $1 \leq i \leq k'$
there is a cycle $C_i'''$ of $F$ which contains all vertices of $B_i
\cup D_i$ and that there are no other cycles in $F$.}

\smallskip

\noindent

Let $C_i''$ be the cycle obtained in the previous claim and let $x$
be a vertex in $B_i$ not covered by $C_i''$. (The case when some vertex in~$D_i$
is not covered by~$C''_i$ is similar.) Let $C$ be the cycle of
$F$ containing $x$ and let $y$ and $z$ be the neighbours of $x$ on $C$.

\medskip

\noindent \textbf{Case 1.} \emph{$y \in A_j$ for some $j$.}

\smallskip

\noindent

It follows from (d) that there are at least $(1 - \eps)m$ vertices of $A_j$ which
have an $H$-neighbour in $B_i$. Also, $y$
has an $H$-neighbour $w$ satisfying the following:
\begin{itemize}
\item[(i)]  both $F$-neighbours of $w$ belong to $A_j \setminus \{y\}$;
\item[(ii)]  both $F$-neighbours of $w$ have an $H$-neighbour in $B_i'$.
\end{itemize}
To see that we can choose such a $w$, suppose first that $A_j =
B_{j'}$ for some $j'$. Then $y$ has a set $N_y$ of at least $\gamma \beta m/8$
$H$-neighbours%
    \COMMENT{don't get $\gamma \beta m/8$ since the cluster size isn't quite $m$}
in $D_{j'}$ by (b). By (a), at most $3\gamma^3 \beta m$ vertices of $N_y$ do not
have both $F$-neighbours in $B_{j'}$. Note that $y$ cannot be one of these
$F$-neighbours in~$B_{j'}$ since $H$ and $F$ are edge-disjoint.
So $N_y$ contains  a set $N_y^*$ of size $\gamma \beta m/9$
so that all vertices in $N_y^*$ satisfy (i). By (d) at most $2 \eps m$ of these do not
satisfy (ii).%
    \COMMENT{need $2\eps m$ instead of $\eps m$ here since one vertex in
$B_{j'}$ having no $H$-neighbour in~$B'_i$ can ruin 2 vertices in~$D_y$}
The argument for the case when $A_j = D_{j'}$ for some
$j'$ is identical.

The next step depends on whether $w$ belongs to
$C_i'', C$ or some other cycle $C'$ of $F$. In all cases we will
find a path $P$ from $x \in B_i$ to a vertex $y'' \in D_i$
containing all vertices of $C_i''\cup C$. We can then proceed as before to
find a cycle containing all the vertices of this path.

\medskip

\noindent \textbf{Case 1a.} \emph{$w \in C_i''$.}

\smallskip

\noindent

Let $y'$ be any one of the $F$-neighbours of $w$. Let $x'$ be any
$H$-neighbour of $y'$ with $x' \in B_i'$ guaranteed by (ii) (so $x'$ lies on $C''_i$) and let $y'' \in D_i$ be the
$F$-neighbour of $x'$ in the segment of $C_i''$ between $x'$ and
$y'$ not containing $w$. Then we can replace the cycles $C_i''$ and
$C$ by the path $xzCywC_i''x'y'C_i''y''$ by removing the edges
$yx,wy'$ and $x'y''$ and adding the edges $yw$ and~$y'x'$.

\medskip

\noindent \textbf{Case 1b.} \emph{$w \in C$.}

\smallskip

\noindent

Let $y'$ be the $F$-neighbour of $w$ in the segment of $C$ between
$y$ and $w$ not containing $x$. Let $x'$ be any $H$-neighbour of
$y'$ with $x' \in B_i'$ and let $y''$ be any $F$-neighbour of $x'$.
Note that $x'$ and $y''$ both lie on $C''_i$ as $x'\in B'_i$.
Then we can replace the cycles $C_i''$ and $C$ by the path
$xzCwyCy'x'C_i''y''$ by removing the edges $yx,wy'$ and $x'y''$ and
adding the edges $yw$ and $y'x'$.

\medskip

\noindent \textbf{Case 1c.} \emph{$w \in C'$ for some $C' \neq
C,C_i''$.}

\smallskip

\noindent

Let $y'$ be any one of the $F$-neighbours of $w$. Let $x'$ be any
$H$-neighbour of $y'$ with $x' \in B_i'$ and let $y''$ be any
$F$-neighbour of $x'$. So $x'$ and $y''$ both lie on $C''_i$.
We can replace the cycles $C_i'',C$ and
$C'$ by the path $xzCywC'y'x'C_i''y''$ by removing the edges $yx,wy'$
and $x'y''$ and adding the edges $yw$ and~$y'x'$.

\medskip

\noindent

\medskip

\noindent \textbf{Case 2.} \emph{$y \in V(X) \setminus (A_1 \cup \cdots \cup A_k)$. }

\smallskip

\noindent

Let $A$ be a cluster so that $y$ has a set $N_y$ of at least $\alpha^2 m$ $H$-neighbours in $A'$
(if $A=B_j$ for some $j$, then $A'$ denotes the set $B_j'$ and similarly if $A=D_j$).
Such an $A$ exists since otherwise $y$ would have at most  $\gamma n + \alpha^2 n$ neighbours
in $H$ by~(\ref{bidi}) and the second part of~(e).
But this would contradict the lower bound of at least $\alpha n/7$ $H$-neighbours given by (e).
Without loss of generality, we may assume that $A=B_j$ for some $j$, the argument for
$A=D_j$ is identical.
Then by (c) there is an index $s\neq j$ so that either ($c_1$) or ($c_2$) holds:
\begin{itemize}
\item[($c_1$)] the pairs $(B_{s},D_{j})$ and $(D_{s},B_i)$ are weakly $(\eps,\eps^3/2k)$-regular in $H$;
\item[($c_2$)] the pairs $(D_{s},D_{j})$ and $(B_{s},B_i)$ are weakly $(\eps,\eps^3/2k)$-regular in $H$.
\end{itemize}
We may assume that ($c_1$) holds, the argument for ($c_2$) is identical.
For convenience, we fix an orientation of each cycle of $F$.
Given a vertex $v$ on a cycle of $F$, this will enable us to refer to
the successor $v^+$ of $v$ and predecessor $v^-$ of $v$.
In particular, let $N_y^+$ be the successors of the vertices in $N_y$ on $C_j''$ and let $N_y^-$ be the predecessors.
So $N_y^+, N_y^- \subseteq D_j$ and $|N_y^-|,|N_y^+| \ge \alpha^2 m$.

Also, let $B_s''$ be the subset of vertices $v$ of $B_s'$ so that both $F$-neighbours
$v^-$ and $v^+$ of $v$ have at least five $H$-neighbours in $B_i'$. Since
$v^-,v^+ \in D_s$, ($c_1$) and~(\ref{bidi}) together imply that $|B_s''| \ge m/2$.
Two application of ($c_1$) to $(B_{s},D_{j})$ now imply that there is a vertex
$w \in N_y$ so that both $w^+$ and $w^-$ have at least one
$H$-neighbour in $B_s''$ (more precisely, apply ($c_1$) to the subpairs $(B_s'',N_y^+)$ and $(B_s'',N_y^-)$).

Suppose first that $C \neq C_j''$. Then let $w_+:=w^+$ and
we can obtain a path $P_1$ with the same vertex set as $C\cup C_j''$
by defining $P_1:=xzCywC_j''w_+$.
If $C=C''_j$, then let $w_+$ be the $C$-neighbour of $w$ on the segment of $C$ between $w$ and $y$
which does not contain $x$ and let $P_1:=xzCwyCw_+$.

Let $v$ be the $H$-neighbour of $w_+$ in $B_s''$ (guaranteed by the definition of $w$).
Note that $v\neq y$ and $v\neq w$ (as $s\neq j$).
Suppose first that $C_s'' \neq C_j'',C$. Then we let $v_+:=v^+$ and define the path $P_2:=xP_1w_+vC_s''v_+$.
If  $C_s'' = C_j''$ or $C_s''=C$, then all vertices of $C_s''$ already lie on $P_1$ and we
let $v_+$ be the $P_1$-neighbour of $v$ on the segment of $P_1$ towards
$w_+$ and let $P_2:=xP_1vw_+P_1v_+$.

Now let $u$ be an $H$-neighbour of $v_+$ in $B_i'$. (To see the existence of $u$, note that $v_+$
is one of th the $F$-neighbours of $v$ in the definition of $B_s''$ since $v\neq w,y$.)
If $C_i'' \neq C_j''$ and $C_i'' \neq C_s''$, then let $u_+:=u^+$ and define the path $P_3:=xP_2 v_+u C_i''u_+$.
If $C_i'' = C_j''$ or $C_i'' = C_s''$, then all the vertices of $C_i''$ already lie on~$P_2$. Since
at most 2 edges of $C''_i$ do not lie on~$P_2$ and since $v_+$ has at least five $H$-neighbours
in $B'_i$ by definition of $B''_s$, we can choose $u$ in such a way that its
$P_2$-neighbours both lie in~$D_i$. We now let $u_+ \in D_i$ be the $P_2$-neighbour of $u$ on the segment
of $P_2$ towards $v_+$ and let $P_3:=xP_2uv_+P_2u_+$.
Note that $P_3$ has endpoints $x\in B_i$ and $u_+ \in D_i$ and contains all vertices of $C''_i\cup C$,
as desired. (We count the whole construction of $P_3$ as one step of the algorithm.)
This completes Case~2.

\medskip

Repeating this procedure, for each $i$ we can find a cycle $C'''_i$ which contains all vertices
of $B_i\cup D_i$. Property~(vi) of Theorem~\ref{Transforming 2-factors into Hamilton
cycles} and the second part of~(e) together imply that no cycle in the $2$-factor $F$ thus obtained
can consist entirely of vertices in $V(X)\setminus (A_1\cup\dots\cup A_k)$ and so
the $C'''_i$ are the only cycles in~$F$.

\medskip

\noindent
\textbf{Claim 5.} \emph{By relabeling if necessary, we may
assume that for every $1 \leq i \leq k'$, the pair $(B_i,D_{i+1})$ is
weakly $(\eps,\eps^3/2k)$-regular in $H$ (where $D_{k' + 1}:=D_1$).}

\smallskip

\noindent

For each $1 \leq i \leq k'$ we relabel $B_i$ and $D_i$ into $D_i$
and $B_i$ respectively with probability $1/2$ independently.
Property (c) together with Theorem~\ref{Chernoff} imply that with
high probability for each $1 \leq i \leq k'$ there are at least $(1
+ \alpha/2)k'/2$ indices $j$ and least $(1+ \alpha/2)k'/2$ indices $j'$
with $1 \leq j,j' \leq k'$ and $j,j' \neq i$ such
that each $(B_i,D_j)$ and each $(B_{j'},D_i)$ are weakly $(\eps,\eps^3/2k)$-regular in~$H$.
Fix such a relabeling. Define a directed graph $J$ on vertex set $[k']$
by joining $i$ to $j$ by a directed edge from $i$ to $j$ if and only
if the pair $(B_i,D_j)$ is weakly $(\eps,\eps^3/2k)$-regular in~$H$. Then
$J$ has minimum out-degree and minimum
in-degree at least $(1 + \alpha/2)k'/2$ and so by
Theorem~\ref{ghouila} it contains a directed Hamilton cycle. Claim~5 now follows
by reordering the indices of the $B_i$'s and $D_i$'s so that they
comply with the ordering in the Hamilton cycle.

\medskip

\noindent \textbf{Claim 6.} \emph{For each $1 \leq j \leq k'$, after
at most $j$ steps, we may assume that $F$ is a union
of cycles together with a path $P_j$ such that~$P_j$ has endpoints $x \in
D_1$ and $y_j \in B_j$, where $y_j$ has an $H$-neighbour in
$D_{j+1}'$, and $P_j$ covers all vertices of $(B_1\cup D_1) \cup \cdots \cup (B_j \cup D_j)$.
Furthermore, for every $j+1 \leq i \leq k'$, either
$P_j$ covers all vertices of $C_i'''$ or $V(P_j)\cap V(C_i''')=\emptyset$.}

\smallskip

\noindent

To prove this claim we proceed by induction on $j$. For the case
$j=1$ observe that by Claim~5 there are at least $(1-\eps)m$ vertices of $B_1$
which have at least one $H$-neighbour in $D_2'$. Of those, there is
at least one vertex $y_1$ which belongs to $B_1'$. Let $x$ be any
$F$-neighbour of $y_1$ (so $x\in D_1$) and remove the edge $xy_1$ from $C_1'''$ to
obtain the path $P_1$. Having obtained the path $P_j$, let $x_{j+1}$
be an $H$-neighbour of $y_j$ in $D_{j+1}'$ (we count the construction of each $P_j$ as
one step of the algorithm).

\medskip

\noindent \textbf{Case 1.} \emph{$P_j$ covers all vertices of
$C_{j+1}'''$.}

\smallskip

\noindent

In this case, let $z_{j+1}$ be the neighbour of $x_{j+1}$ on~$P_j$ in the
segment of $P_j$ between $x_{j+1}$ and $y_j$ and let $Q_{j+1}$ be
the path obtained from $P_j$ by adding the edge $y_jx_{j+1}$ and
removing the edge $x_{j+1}z_{j+1}$. Observe that the endpoints of
the path are $x \in D_1$ and $z_{j+1} \in B_{j+1}$ (but $z_{j+1}$ need not have
an $H$-neighbour in $D'_{j+2}$). By~(a) and~(b) $z_{j+1}$ has at least $\gamma
\beta m/8$ $H$-neighbours $w_{j+1}$ in $D'_{j+1}$. For each such
$H$-neighbour $w_{j+1}$, let $w'_{j+1}$ be
the unique neighbour of $w_{j+1}$ on $Q_{j+1}$ in the segment of $Q_{j+1}$
between $w_{j+1}$ and $z_{j+1}$. So $w'_{j+1}\in B_{j+1}$. Since
by the previous claim at most $\eps m$ vertices of $B_{j+1}$ do not
have an $H$-neighbour in $D_{j+2}'$, we can choose a $w_{j+1}$ so
that $w'_{j+1}$ has an $H$-neighbour in
$D_{j+2}'$. We can then take $y_{j+1}:=w'_{j+1}$ and obtain
$P_{j+1}$ from $Q_{j+1}$ by adding the edge $z_{j+1}w_{j+1}$ and
removing the edge $w_{j+1}w'_{j+1}$.

\medskip

\noindent \textbf{Case 2.} $V(P_j)\cap V(C_{j+1}''')=\emptyset$.

\smallskip

\noindent

In this case, we let $z_{j+1}$ be any $F$-neighbour of $x_{j+1}$ and
let $Q_{j+1}$ be the path obtained from $P_j$ and $C_{j+1}'''$ by
adding the edge $y_jx_{j+1}$ and removing the edge $x_{j+1}z_{j+1}$.
Observe that the endpoints of the path are $x \in D_1$ and $z_{j+1}
\in B_{j+1}$ and so this case can be completed as the previous
case.

\medskip

By the case $j=k'$ of the previous claim we may assume that we
now have a path $P:=P_{k'}$ which covers all vertices of $A_1 \cup
\cdots \cup A_k$ and has endpoints $x \in D_1$ and $y:=y_{k'} \in
B_{k'}$ where $y$ has an $H$-neighbour in $D_1'$. Moreover, $P$ contains all
vertices of each $C'''_i$ and so by Claim~4 it must be a Hamilton path.
Now let $z$ be any
$H$-neighbour of $y$ with $z \in D_1'$ and let $w$ be the neighbour
of $z$ in the segment of $P$ between $z$ and $y$. Let $Q$ be the
path obtained from $P$ by removing the edge $wz$ and adding the edge
$yz$. So $Q$ is a path on the same vertex set as $P$ with endpoints $x \in D_1$ and
$w \in B_1$ (we count the construction of $Q$ as another step of the algorithm).
But then we can apply Lemma~\ref{Rotation-Extension} to
transform $Q$ into a Hamilton cycle in one more step, thus completing the proof of
Theorem~\ref{Transforming 2-factors into Hamilton cycles}.
\end{proof}

\section{Acknowledgment}
We would like to thank Andrew Treglown for helpful discussions.

\medskip

{\footnotesize \obeylines \parindent=0pt

\begin{tabular}{lll}

Demetres Christofides               &\ &  Daniela K\"{u}hn \& Deryk Osthus \\
School of Mathematical Sciences     &\ &  School of Mathematics \\
Queen Mary, University of London    &\ &  University of Birmingham \\
Mile End Road                       &\ &  Edgbaston \\
London                              &\ &  Birmingham \\
E1 4NS                              &\ &  B15 2TT \\
UK                                  &\ &  UK \\

\end{tabular}
}

{\footnotesize \parindent=0pt

\it{E-mail addresses}:
\tt{d.christofides@qmul.ac.uk}, \tt{\{kuehn,osthus\}@maths.bham.ac.uk}}

\end{document}